\documentclass{article}
\usepackage[numbers]{natbib}
\usepackage[utf8]{inputenc}
\usepackage{amsmath}
\usepackage{amsfonts}
\usepackage{amssymb}
\usepackage{mathtools}
\usepackage{mathrsfs}
\usepackage{hyperref}
\usepackage{xcolor}
\usepackage{amsthm}
\usepackage{framed}
\usepackage{manfnt} 
\usepackage [english]{babel}
\usepackage [autostyle]{csquotes}
\usepackage{parskip}
\usepackage{tikz-cd} 
\usepackage[right = 3cm, left=3cm, top = 3cm, bottom = 3cm]{geometry}
\usepackage[pdftex]{pict2e}
\usepackage{graphicx}
\usepackage{comment}

\usepackage{appendix}

\colorlet{shadecolor}{gray!50}
\definecolor{darkspringgreen}{rgb}{0.13, 0.55, 0.35}

\newcommand{\Z}{\mathbb{Z}}

\newcommand{\A}{\mathbb{A}}
\newcommand{\Ps}{\mathbb{P}}

\newcommand{\etale}{\'etale }

\newcommand{\Einf}{$E_{\infty}$}

\newtheorem{definition}{Definition}[section]
\newtheorem{theorem}[definition]{Theorem}
\newtheorem{lemma}[definition]{Lemma}
\newtheorem{proposition}[definition]{Proposition}
\newtheorem{corollary}[definition]{Corollary}

\newtheorem{example}[definition]{Example}

\theoremstyle{definition}

\newtheorem{remark}[definition]{Remark}

\title{A priori obstructions to resolution of singularities}
\author{Tobias Shin}
\date{}

\begin{document}

\maketitle
\begin{abstract}
    We present an argument due to Thom to formulate a priori cohomology obstructions for a projective variety to admit an embedded resolution of singularities, and generalize the argument to a field of characteristic $p > 0$. We show that these obstructions are defined for a general cohomology theory equipped with a proper pushforward, and which satisfies localization, \'etale excision, and $\mathbb{A}^1$-invariance. As examples, we show that odd degree Steenrod homology operations on higher Chow groups mod $\ell$, defined by a suitable $E_{\infty}$-algebra structure, vanish on the fundamental class of a smooth variety, where $\ell$ is any prime. We extend this general vanishing result to arbitrary varieties for operations defined mod $\ell \neq p$. Along the way, we also obtain a Wu formula for the mod $p$ Steenrod operations in the case of closed embeddings of smooth varieties.
\end{abstract}

\tableofcontents

\section{Introduction}

The aim of this paper was to na\"ively define topological obstructions for a variety to admit a resolution of singularities over an algebraically closed field $k$ of characteristic $p > 0$. We do not succeed in this endeavor; instead, we present a vanishing theorem for generalized cohomology theories satisfying suitable axioms. As an example, we can consider motivic cohomology with coefficients mod $\ell$ where $\ell$ is any prime (including $\ell = p$). There exist mod $\ell$ Steenrod operations $Q^\bullet$ on motivic cohomology, arising from an \Einf-algebra structure, by work of \cite{joshua2001motivic}, \cite{brosnan2003steenrod}, and  \cite{kriz1995operads}. We let $\beta$ denote the Bockstein and its dual.

\begin{theorem}\label{Theorem1.1}
Let $X \hookrightarrow P$ be a closed embedding of a variety $X$ into a smooth variety $P$ over an algebraically closed field of characteristic $p > 0$. Let $Q^\bullet$ be the simplicial Steenrod operations on motivic cohomology $H^*_X(P,\Z/\ell(*))$, where $\ell \neq p$, and let $Q_\bullet$ denote the dual operations on $H_*(X,\Z/\ell(*))$. The odd degree operations $\beta Q_s$ vanish on the fundamental class $[X]$. If $X$ is also assumed smooth, then the odd degree operations $\beta Q_s$ vanish on the fundamental class $[X]$ for $\ell = p$.
\end{theorem}

Theorem \ref{Theorem1.1} is a special case of Theorem \ref{Theorem1.2}, which applies to \textbf{generalized cohomology theories} in the sense of Panin--Smirnov (see Definition \ref{generalizedcohomologytheory}). Here, a cohomology theory $A$ is a contravariant functor that has a long exact sequence of a pair, satisfies \etale excision, and is $\mathbb{A}^1$-invariant. Moreover we require our theory $A$ to be equipped with a \textbf{proper pushforward} (Definition \ref{generalizedproperpushforwarddefinition}). It is then a theorem of Panin--Smirnov that this is equivalent to equipping $A$ with a notion of Chern classes, which is also equivalent to equipping $A$ with a notion of Thom class and Thom isomorphism.

\begin{theorem}\label{Theorem1.2}
    Let $A$ be a bigraded ring cohomology theory in the sense of Panin--Smirnov, valued in modules over $\Z/\ell$ where $\ell$ is any prime, and equipped with a proper pushforward. Suppose $A^{i,j}(\textnormal{Spec }k) = 0$ for $i \neq 0$. Let $\phi = \phi^\bullet$ be a ring operation with well-defined Todd genus that preserves even degrees in the first index, and let $\beta$ be a natural operation that sends bidegrees $(i,j)$ to $(i+1,j)$ that commutes with proper pushforward. Moreover, assume $\beta$ is a derivation that vanishes on the Todd genus of any bundle. Suppose $X \hookrightarrow P$ is a closed embedding of a smooth variety $X$ into a smooth variety $P$. The operation $\beta\phi^s$ vanishes on the dual class $\tau \in A^{2c,c}_{X}(P)$ where $c$ is the codimension of $X$ in $P$.
\end{theorem}

\begin{theorem}\label{Theorem1.3}
    Let $A$ be a theory with the same hypotheses as in Theorem \ref{Theorem1.2}. Suppose $X \hookrightarrow P$ is a closed embedding of a (possibly \textbf{singular}) variety $X$ into a smooth variety $P$ over a field of characteristic $p > 0$. If $A$ is valued in $\Z/\ell$-modules, where $\ell \neq p$, then the operation $\beta\tilde{\phi}^s$ vanishes on the dual class $\tau \in A^{2c,c}_{X}(P)$, where $\tilde{\phi}$ is the operation associated to $\phi$ that commutes with proper pushforward. Moreover, if $A$ is valued in $\Z/p$-modules and $X$ admits an embedded resolution of singularities, then the operation $\beta\tilde{\phi}^s$ vanishes on the dual class $\tau \in A^{2c,c}_{X}(P)$.
\end{theorem}

Theorem \ref{Theorem1.3} essentially defines a priori cohomological obstructions to a projective variety $X$ admitting an embedded resolution of singularities. If $X$ admits a resolution, then certain mod $p$ cohomological operations must vanish. Thus, in theory, one could construct an example of a singular variety $X$ where these operations are non-zero. As of this writing, and to the author's knowledge, the problem of resolution of singularities in positive characteristic is still open; if proven true, then Theorems \ref{Theorem1.1}, \ref{Theorem1.2}, and \ref{Theorem1.3} can be interpreted as vanishing theorems.

However, the only examples we have either already vanish for reasons other than resolution of singularities, or do not satisfy the hypotheses of the theorems: the mod $\ell$ operations on motivic and \etale cohomology vanish due to alterations, and we do not know how to push forward the mod $p$ operations.

Our motivation begins with the work of Thom \cite{thom1952espaces}. In private communication, Sullivan informed the author of his belief that Thom, after his work on Steenrod's problem of representing homology classes by manifolds, was motivated to formulate cohomological obstructions to resolving the singularities of a compact complex algebraic variety; Sullivan explains this belief explicitly in the article \cite{sullivan2004rene}. The cohomology obstructions are defined as follows: first, for a given compact complex algebraic variety $X$, one topologically embeds $X$ into a large euclidean space $\mathbb{R}^N$. There are mod $\ell$ Steenrod operations $Q^\bullet$ defined on the cohomology of a pair $(\mathbb{R}^N, \mathbb{R}^N-X)$, for any prime $\ell$. By Alexander--Lefschetz--Poincar\'e duality, these cohomology operations dualize to mod $\ell$ operations $Q_\bullet$ on the (Borel--Moore) homology of $X$. These operations do not depend on the embedding or dimension of euclidean space. Moreover, these homology operations commute with proper pushforwards. Sullivan then explained the following argument, which he attributed to Thom. We abuse notation and let $\beta$ denote the dualized Bockstein operation on homology.

\begin{proposition}\label{mainnovelty}
    Let $X$ be a compact complex algebraic variety. Then the operations of odd degree $\beta Q_\bullet$ vanish on the fundamental class $[X]$.
\end{proposition}
\begin{proof}
    We first prove the case when $X$ is smooth. We smoothly embed $X$ into a large euclidean space $\mathbb{R}^n$. Since $X$ is complex, its normal bundle is stably almost complex, so $N_X\mathbb{R}^n \oplus \mathbb{R}^k$ is a complex bundle for some $k$. Embed $X$ into the larger euclidean space $\mathbb{R}^N$ where $N = n+k$, so that its normal bundle is complex. By definition of the operations, it suffices to show $\beta Q^\bullet$ vanishes on the dual class $\tau$ in the relative cohomology $H^*(\mathbb{R}^N,\mathbb{R}^N-X,\Z/\ell)$. We have a natural isomorphism $H^*(\mathbb{R}^N,\mathbb{R}^N-X,\Z/\ell) \cong H^*(N_X\mathbb{R}^N,N_X\mathbb{R}^N-0,\Z/\ell)$ where $N_X\mathbb{R}^N$ denotes the normal bundle of $X$ with respect to our embedding. In this equivalence, the dual class $\tau$ is identified with the Thom class of the normal bundle, so we will use the same notation for both.

    The normal bundle $N_X\mathbb{R}^N$ induces a classifying map $X \rightarrow Gr(c)$ where $Gr(c)$ is the infinite Grassmannian of complex $c$-dimensional subspaces, where $c$ is the complex codimension of $X$ in $\mathbb{R}^N$.

    \begin{center}
\begin{tikzcd}
N_XP \arrow[r] \arrow[d]
& \gamma^c \arrow[d] \\
X \arrow[r]
& Gr(c)
\end{tikzcd}
\end{center}
    where $\gamma^c$ is the tautological bundle of complex $c$-dimensional subspaces. The classifying map induces a map on the Thom spaces, and so a map $H^*(\gamma^c, \gamma^c-0,\Z/\ell) \rightarrow H^*(N_X\mathbb{R}^N,N_X\mathbb{R}^N-0,\Z/\ell)$, and sends a cohomology class $\gamma$ in cohomology of the the universal Thom space to the Thom class $\tau$ of the normal bundle, by functoriality of the Thom isomorphism. The cohomology of the complex Grassmannian $Gr(c)$ is concentrated in even degrees, and so operations $\beta Q^\bullet$ vanish on $\gamma$. By naturality, they then vanish on the Thom class $\tau$, and so the dual homology operations $\beta Q_\bullet$ vanish on the fundamental class $[X]$.

    Suppose now $X$ is singular. Then $X$ has a resolution of singularities $Y$ by Hironaka \cite{hironaka1964resolution}, where $Y$ is smooth and the map $Y \rightarrow X$ is proper and birational. In particular, the fundamental class $[Y]$ is sent to $[X]$. The odd degree operations vanish on $[Y]$ by the argument above, and they commute with proper pushforward, so they also vanish on $[X]$.
\end{proof}

If $X$ is a singular variety that admits a resolution, then the odd degree homology operations must vanish on its fundamental class. Thus, these operations are ``obstructions'' in the sense that, finding an example of a variety where these operations are non-zero would give a variety that had no resolution. They are ``a priori'' as they do not require resolution of singularities to be defined.

The idea of this paper is to translate this differential topology argument into algebraic geometry, over a field of characteristic $p > 0$. The central difficulty was, and still may be, finding the correct cohomology theory suitable for this program. Initially, Sullivan considered using \etale cohomology \cite{sullivan2004rene} but learned from de Jong \cite{de1996smoothness} that a priori mod $\ell \neq p$ obstructions would vanish. The argument then required a suitable $p$-adic cohomology theory, which captured $p$-torsion information. There is a large amount of literature in the field exploring the tension between a ``nice" $p$-adic cohomology theory, Steenrod operations, and resolution of singularities (for example, see the articles \cite{haution2012reduced}, \cite{ertl2021integral}, \cite{abe2021integral}, \cite{merici2022motivic}, \cite{niziol2006p}).

By work of Joshua \cite{joshua2001motivic} and Kriz--May \cite{kriz1995operads}, there are Steenrod operations $Q^\bullet$ on motivic cohomology mod $\ell$ for smooth varieties, where $\ell$ is any prime (including the characteristic). The operations are derived from the product structure on Bloch's cycle complexes, which is only well defined up to quasi-isomorphism. This gives the cycle complexes an \Einf-algebra structure, which then inherits generalized Steenrod operations by work of May \cite{may1970general}. We emphasize here that these Steenrod operations are \textbf{not} the same as those defined by Voevodsky \cite{voevodsky2003reduced} or Primozic \cite{primozic2020motivic}, as those operations vanish on fundamental classes for trivial degree reasons (see Proposition \ref{motivicOpsVanishForDegreeReasons}).

Once the operations are defined for cohomology with support however, there is now another nuance; we thank Patrick Brosnan for informing us of this crucial point. In classical topology, one can define the Borel--Moore homology groups of a space $X$ as the Alexander--Lefschetz--Poincar\'e dual of the cohomology of a pair $(\mathbb{R}^N,\mathbb{R}^N-X)$. The groups are then proved to be independent of the choice of embedding. The cohomology of a pair admits Steenrod operations; however, the dualized homology operations \textbf{do} depend on the embedding and ambient space. In particular, the operations \textbf{do not} commute with proper pushforward. At best, there is a Riemann--Roch type formula that describes their failure to commute with pushforward, with error terms arising from characteristic classes of the tangent bundle of the ambient space. We call these error terms a \textbf{Todd genus} (see Definition \ref{toddgenusdefinition}). Thus, if one embeds into euclidean space, the error terms all vanish. The work of Panin--Smirnov \cite{panin2000oriented}, \cite{panin2003oriented}, \cite{panin2004riemann}, \cite{panin2009oriented}, \cite{smirnov2007riemann} gives a general Riemann--Roch type formula for generalized cohomology theories in the sense above, given a notion of proper pushforward on cohomology (see Definition \ref{generalizedproperpushforwarddefinition}) and a suitable ring operation $\phi$ (see Definition \ref{ringcohomologydefinition}). Given such a Riemann--Roch type formula, one can twist the operation by the Todd genus, and we can define new operations $\tilde{\phi}$ that do commute with pushforward. It is through these Riemann--Roch type formulas that the geometry of the ambient space is crucial to our homology operations. This is the content of Section \ref{RiemannRochTypeFormulasSection}.

This paper is then about mimicking the argument of Thom in Proposition \ref{mainnovelty}. Moreover, once this is done, one can use alterations of de Jong \cite{de1996smoothness} with appropriate degree \cite{temkin2017tame} to extend the vanishing results to singular varieties, so long as the coefficients are valued mod $\ell \neq p$. The proofs of the Theorems \ref{Theorem1.1}, \ref{Theorem1.2}, and \ref{Theorem1.3} are the content of Sections \ref{embeddingBundlesSection} and \ref{ThomObsVanish}.

As mentioned prior, we do not know how to push forward the mod $p$ motivic operations, as there is no Riemann--Roch type formula. Although there is no general Riemann--Roch type formula for the mod $p$ operations for general proper pushforwards, we are able to obtain one for the special case of closed embeddings of smooth varieties. This is proven as Theorem \ref{WuformulaforSimplicialSteenrodOps}.

\begin{theorem} (Wu formula)
    Let $i: X \hookrightarrow P$ be a closed embedding of smooth varieties, and let $i_!$ denote the proper pushforward on mod $p$ motivic cohomology. We obtain the formula
    \begin{equation*}
        Q^\bullet i_!(\alpha) = i_!(Q^\bullet(\alpha) \smile c_{\textnormal{top}}^{p-1}(N_XP))
    \end{equation*}
    for $\alpha \in H^*(X,\Z/p(*))$, where $Q^\bullet$ are the mod $p$ simplicial operations on mod $p$ motivic cohomology, and $c_{\textnormal{top}}(N_XP)$ is the top Chern class of the normal bundle of $X$ in $P$.
\end{theorem}

Finally, we would like to say that our approach is na\"ive, in that we have given a set of axioms sufficient but not necessary for Thom's argument to work. For example, one could relax the homotopy invariance axiom and still obtain many of the formal results needed for the argument. 


We also want to note the similiarity in spirit of our work with that of Fulton (Example 19.1.8 in \cite{fulton2013intersection}), Kawai \cite{kawai1977note}, Benoist \cite{benoist2022steenrod}, Quick \cite{quick2011torsion}, and Atiyah--Hirzebruch \cite{atiyah1961cohomologie} who all utilize Steenrod operations as obstructions to cohomology classes being represented by subvarieties (for example, disproving the integral Hodge conjecture).

\textbf{Structure of the paper.} The notion of generalized cohomology theory is given in section \ref{AxiomaticApproach}, and the Riemann--Roch type theorems of Panin--Smirnov are discussed in Section \ref{RiemannRochTypeFormulasSection}. Section \ref{embeddingBundlesSection} is a brief digression on stacks and the cohomology of the classifying stack of vector bundles, and section \ref{ThomObsVanish} brings all aforementioned sections together to generalize the Thom argument for generalized cohomology theories. It also proves the Thom argument for certain Steenrod operations defined on motivic cohomology. Section \ref{alterationsSection} shows the mod $\ell \neq p$ operations all vanish.

\textbf{Notation.} Throughout this paper, we fix an algebraically closed field $k$ of characteristic $p > 0$. A variety $X$ is an integral, separated scheme over $k$ of finite type. All varieties are quasiprojective unless specified otherwise. All products and dimensions of spaces are over the field $k$. Likewise, all tensor products of abelian sheaves are over $\Z$ or $\mathbb{Z}/\ell$, where $\ell$ is any prime. 

\textbf{Acknowledgements.} We would like to thank Dennis Sullivan, who originally allowed the author to work on this project as his thesis, and for telling us the Thom argument. We would like to also thank Thomas Geisser and Marc Levine for patiently answering the author's basic questions on higher Chow groups, Patrick Brosnan for teaching us that the dual Steenrod homology operations on Borel--Moore homology do not commute with proper pushforward, Peter May for helping the author understand some nuances about \Einf-algebras, and Spencer Bloch for clarifying some points in his private letter. We also want to thank Benson Farb for helpful comments on an initial draft of this paper. The author would also like to thank Toni Annala, Connor Dube, Kevin Lin, Akhil Mathew, and Callum Sutton for teaching him about stacks for the second draft of this paper. We also thank the anonymous referee for comments drastically improving the exposition.

\section{Generalized cohomology, proper pushforwards, and ring operations}\label{Steenrod}

\subsection{Cohomology theories and integrations with support}\label{AxiomaticApproach}
In this section we introduce the axioms of a generalized cohomology theory in the sense of Panin--Smirnov. Their work concerns ring cohomology theories equipped with an \textit{orientation} (i.e., a Thom class). They show that this is equivalent to equipping a given ring cohomology theory with a \textit{Chern structure} (i.e., Chern classes), and also equivalent to equipping it with a \textit{transfer} (i.e., a proper pushforward) or a \textit{trace} (i.e., a Thom isomorphism). Given a ring cohomology theory equipped with any one of these equivalent notions, and given a natural ring operation on the cohomology theory, they are then able to deduce Grothendieck--Riemann--Roch type formulas describing the failure of commutativity for the ring operation with proper pushforward. Most of the material in this section can be found in the papers of Panin \cite{panin2009oriented}, Panin--Smirnov \cite{panin2004riemann}, Smirnov \cite{smirnov2007riemann}, and Levine \cite{levine2008oriented}.

\subsubsection{Cohomology and pushforward \`a la Panin, Smirnov, Levine}

We give the definition of a ring cohomology theory, following Panin \cite{panin2009oriented}.
\begin{definition}\label{generalizedcohomologytheory} Let \textbf{SmOp} denote the category of pairs $(X,U)$ where both $X$ is a smooth variety and $U$ is open in $X$. Morphisms are maps of pairs. The category \textbf{Sm} of smooth varieties is a full subcategory of \textbf{SmOp} by sending $X$ to the pair $(X,\varnothing)$.

A \textbf{cohomology theory} is a contravariant functor $\textbf{SmOp} \xrightarrow{A} \textbf{Ab}$ with natural maps $A(U) \xrightarrow{\partial} A(X,U)$ satisfying the following properties:

\begin{enumerate}
    \item (Localization) For a pair $(X,U) \in \textbf{SmOp}$, the sequence 
    \begin{equation*}
        A(X) \xrightarrow{j^*} A(U) \xrightarrow{\partial_{(X,U)}} A(X,U) \xrightarrow{i^*} A(X) \xrightarrow{j^*} A(U) 
    \end{equation*}
    is exact, where $j: U \hookrightarrow X$ and $i: (X,\varnothing) \hookrightarrow (X,U)$ are the natural inclusions.

    \item (Excision) The map $A(X,U) \rightarrow A(X',U')$ induced by a morphism $e: (X',U') \rightarrow (X,U)$ is an isomorphism, if the map $e: X' \rightarrow X$ is \'etale, and for $Z = X-U$ and $Z' = X'-U'$, one has $e^{-1}(Z) = Z'$ and $e: Z' \rightarrow Z$ is an isomorphism.

    \item (Homotopy) The map $A(X) \rightarrow A(X \times \A^1)$ induced by the projection $X \times \A^{1} \rightarrow X$ is an isomorphism.
\end{enumerate}
The operator $\partial$ is the \textbf{boundary operator}. A morphism of cohomology theories $\phi: (A,\partial^A) \rightarrow (B,\partial^B)$ is a natural transformation $\phi: A \rightarrow B$ such that for every pair $(X,U) \in \textbf{SmOp}$, the boundary operators commute with the natural transformation maps. That is, we have $\partial^B\circ \phi_U = \phi_{(X,U)} \circ \partial^A$.

We also write $A_Z(X)$ for $A(X,U)$ where $U = X-Z$ for $Z$ a closed subset of $X$, and call $A_Z(X)$ the cohomology of $X$ with support on $Z$.
\end{definition}


\begin{definition}\label{ringcohomologydefinition}
    Let $P = (X,U)$ and $Q = (Y,V)$ be in \textbf{SmOp}. We define their product to be $P \times Q = (X \times Y, X\times V \smile U \times Y) \in \textbf{SmOp}$.

    A \textbf{ring cohomology theory} $A$ is a cohomology theory such that, for every $P, Q \in \textbf{SmOp}$, there is a natural bilinear morphism $\times: A(P) \times A(Q) \rightarrow A(P \times Q)$ which we call the \textbf{external product} functorial in both variables with the following properties:

    \begin{enumerate}
        \item (Associativity) $(a\times b) \times c = a \times (b\times c) \in A(P \times Q \times R)$ for $a \in A(P)$, $b\in A(Q)$, and $c \in A(R)$;

        \item (Unit) There is an element $1 \in A(\textnormal{Spec } k)$ such that for any pair $P \in \textbf{SmOp}$ and any $a \in A(P)$, one has $1 \times a = a = a \times 1 \in A(P)$;

        \item (Partial Leibniz) $\partial_{P \times Y}(a\times b) = \partial_{P}(a) \times b \in A(X \times Y, U \times Y)$ for a pair $P = (X,U) \in \textbf{SmOp}$, a smooth variety $Y$, and elements $a\in A(U), b\in A(Y)$.
    \end{enumerate}

    Given an external product, one can define cup products $A_Z(X) \times A_{Z'}(X) \rightarrow A_{Z\cap Z'}(X)$ by defining $a \smile b = \Delta^*(a \times b)$, where $\Delta: (X, U \smile V) \hookrightarrow (X \times X, X \times V \smile U \times X)$ is the diagonal. If $p: X \rightarrow \textnormal{Spec } k$ is the structure map, then $p^*(1) \in A(X)$ is the unit for the cup products $A_Z(X) \times A(X) \rightarrow A_Z(X)$ and $A(X) \times A_Z(X) \rightarrow A_Z(X)$. A partial Leibniz rule for cup product also holds: $\partial(a \smile j^*b) = \partial a \smile j^*b$ for $a \in A(U), b\in A(X)$.

    Moreover, given a cup product, one can construct an external product $a\times b = \pi_X^*(a) \smile \pi_Y^*(b)$ for $a \in A(X,U)$ and $b \in A(Y,V)$, where $\pi_X$ and $\pi_Y$ are the natural projections from $X \times Y$. This construction is inverse to the construction above.
\end{definition}

The two primary examples to keep in mind are motivic cohomology and mod $\ell \neq p$ \etale cohomology. Other examples of ring cohomology theories can be found in the paper of Panin \cite{panin2003oriented} such as classical singular cohomology over $\mathbb{C}$, algebraic $K$-theory, algebraic cobordism theory, any representable theory, etc.


\begin{definition}
    Suppose we have a fiber product diagram
\begin{center}
\begin{tikzcd}
X' \arrow[r, "i'"] \arrow[d, "f'"]
& X \arrow[d, "f"] \\
Y' \arrow[r, "i"]
& Y
\end{tikzcd}
\end{center}
    where all varieties are smooth, $Y' \xhookrightarrow{i} Y$ is a closed embedding, with normal bundles $N := N_{Y/Y'}$ and $N' := N_{X/X'}$. The square is \textbf{transversal} if the canonical morphism $N' \rightarrow f'^*N$ is an isomorphism.
\end{definition}

We now present the definition of an integration for a ring cohomology theory, following Levine \cite{levine2008oriented} and Panin \cite{panin2009oriented}. Levine works in the setting where his cohomology theory is $\Z/2$-graded; however, his results hold for bigraded theories with the assumption that integrations shift bi-degrees $f_!: A^{p,q}(X) \rightarrow A^{p+2d,q+d}(Y)$, where $d = \textnormal{dim }(Y) - \textnormal{dim }(X)$. See Remark 1.7 in \cite{levine2008oriented}.

\begin{definition}\label{generalizedproperpushforwarddefinition}
    Let $A$ be a ring cohomology theory. A \textbf{proper pushforward} (or \textbf{integration} or \textbf{trace structure} or \textbf{orientation} or \textbf{transfer}) is an assignment $(f: X \rightarrow Y) \mapsto f_!: A(X) \rightarrow A(Y)$ where $f_!$ is a two-sided $A(X)$-module map, and $f$ is a projective morphism between smooth varieties $X$ and $Y$, such that the following properties are satisfied: 

    \begin{enumerate}
        \item $(g\circ f)_! = g_!\circ f_!$ for projective morphisms $f: X \rightarrow Y$ and $g: Y \rightarrow Z$;
        \item for every transversal square, the following diagram commutes:
        \begin{center}
        \begin{tikzcd}
        A(X') \arrow[r, "i'_!"] 
        & A(X) \\
        A(Y') \arrow[r, "i_!"] \arrow[u, "f'^*"]
        & A(Y);  \arrow[u, "f^*"]
        \end{tikzcd}
        \end{center}

        \item for any morphism of smooth varieties $f: X \rightarrow Y$, the following diagram commutes:
        \begin{center}
        \begin{tikzcd}
        A(\Ps^n \times X) \arrow[r, "(\pi_X)_!"] 
        & A(X) \\
        A(\Ps^n \times Y) \arrow[r, "(\pi_Y)_!"] \arrow[u, "(id\times f)^*"]
        & A(Y),  \arrow[u, "f^*"]
        \end{tikzcd}
        \end{center}
        where $\pi_X: \Ps^n \times X \rightarrow X$ and $\pi_Y: \Ps^n \times Y \rightarrow Y$ are the natural projections;

        \item for any smooth variety $X$, one has $id_! = id^*$ on $A(X)$;

        \item for any closed embeddings of smooth varieties $X \xhookrightarrow{i} P$, with complementary open embedding $P-X \xhookrightarrow{j} P$, the sequence $A(X) \xrightarrow{i_!} A(P) \xrightarrow{j^*} A(P-X)$ is exact.
    \end{enumerate}
\end{definition}


For sake of completion, we also present the notion of an integration with support, due to Levine \cite{levine2008oriented}, following Panin \cite{panin2004riemann}. This is essentially Definition \ref{generalizedproperpushforwarddefinition}, with supports added everywhere.

\begin{definition}
    Let $A$ be a ring cohomology theory. We define a new category \textbf{SmOp'} where objects are pairs $(X,U)$ where $U$ is open in smooth $X$, and morphisms $f: (X,U) \rightarrow (Y,V)$ are projective maps $f: X \rightarrow Y$ such that $f(X-U) \subset Y-V$. Equivalently, we can consider pairs $(X,Z)$ where $Z$ is a closed subset of $X$ (called the support), and consider projective maps $f: X \rightarrow Y$ that send supports to supports.
    
    An \textbf{integration with supports} is an assignment to every morphism $f: (X,Z) \rightarrow (Y,W)$ in \textbf{SmOp'} a map $f_!: A_Z(X) \rightarrow A_W(Y)$ satisfying the following properties:

    \begin{enumerate}
        \item $(g\circ f)_! = g_!\circ f_!$ for composable morphisms;
        \item for $f: (X,Z) \rightarrow (Y,W)$ in \textbf{SmOp'}, and $S$ a closed subset of $Y$, $f_!$ is a $A_S(Y)$-module map; i.e., the diagram
        \begin{center}
        \begin{tikzcd}
        A_S(Y) \otimes A_Z(X)  \arrow[r, "f^*\smile"] \arrow[d, "id \otimes f_!"]
        & A_{Z\cap f^{-1}(S)}(X) \arrow[d, "f_!"] \\
        A_S(Y) \otimes A_W(Y) \arrow[r, "\smile"]
        & A_{S\cap W}(Y)
        \end{tikzcd}
        \end{center}
        commutes;
        \item let $i: (X,Z) \rightarrow (Y,W)$ be in \textbf{SmOp'} such that $i: X \rightarrow Y$ is a closed embedding, and let $g: (Y',W') \rightarrow (Y,W)$ be a morphism in \textbf{SmOp}. Let $X' = X \times_Y Y'$, $g': X' \rightarrow X$ and $i': X' \rightarrow Y'$ be the projections. Suppose in addition that $X'$ is smooth and the square 
        \begin{center}
        \begin{tikzcd}
        X' \arrow[r, "i'"] \arrow[d, "g'"]
        & Y' \arrow[d, "g"] \\
        X \arrow[r, "i"]
        & Y
        \end{tikzcd}
        \end{center}
        is transverse. Then the diagram
        \begin{center}
        \begin{tikzcd}
        A_{i'^{-1}(W)}(X') \arrow[r, "i'_!"] 
        & A_{W'}(Y')  \\
        A_Z(X) \arrow[r, "i_!"]\arrow[u, "g'^*"]
        & A_W(Y)\arrow[u, "g^*"]
        \end{tikzcd}
        \end{center}
        commutes;

        \item let $f: (X,Z) \rightarrow (Y,W)$ be a morphism in \textbf{SmOp} and let $p_X: \Ps^n \times X \rightarrow X$ and $p_Y: \Ps^n \times  Y \rightarrow Y$ be the projections. Then the diagram 
        \begin{center}
        \begin{tikzcd}
        A_{\Ps^n \times Z}(\Ps^n \times X)  \arrow[d, "(p_X)_!"]
        & A_{\Ps^n \times W}(\Ps^n \times Y) \arrow[l, "(id \times f)^*"]  \arrow[d, "(p_Y)_!"] \\
        A_Z(X) 
        & A_W(Y)\arrow[l, "f^*"]
        \end{tikzcd}
        \end{center}
        commutes;

        \item given $(X,Z)$ and $(X,Z')$ such that $Z' \subset Z$, the maps $id_*$ and $id^*$ from $A_{Z'}(X) \rightarrow A_{Z}(X)$ agree (here, the pull-back is induced by the long exact sequence of a triple, see Remark 1.4 in Levine \cite{levine2008oriented});

        \item let $f: X \rightarrow Y$ be a projective morphism of smooth varieties, let $S \subset S' \subset X$ and $T \subset T' \subset Y$ be closed subsets, and suppose $f^{-1}(T) \cap S' = S$, and $f(S') \subset T'$. Then the diagram        \begin{center}
        \begin{tikzcd}
        A_{S'-S}(X-S) \arrow[r, "\partial_{X,S',S}"] \arrow[d, "f_!"]
        & A_{S}(X) \arrow[d, "f_!"] \\
        A_{T'-T}(Y-T) \arrow[r, "\partial_{Y,T',T}"]
        & A_{T}(Y)
        \end{tikzcd}
        \end{center}
        commutes. Here, $\partial_{X,S',S}$ and $\partial_{Y,T',T}$ are the boundary maps in the long exact sequences for the triples $(X,S',S)$ and $(Y,T',T)$, and the pushforward map $f: A_{S'-S}(X-S) \rightarrow A_{T'-T}(Y-T)$ is the composition
        \begin{equation*}
            A_{S'-S}(X-S) \xrightarrow{j^*} A_{S'-S}(X-f^{-1}(T)) \xrightarrow{(f|_{X-f^{-1}})_*} A_{T'-T}(Y-T).
        \end{equation*}
    \end{enumerate}
\end{definition}


Given an integration, one can construct an integration with supports uniquely due to the following theorem of Levine (Theorem 1.12 and Corollary 1.13 in \cite{levine2008oriented}). Although stated for the $\Z/2$-graded theories, the proof holds for the $\Z$-graded and bigraded cases similarly.

\begin{theorem}\label{levineIntegrationIsUnique} (Levine) Let $A$ be a ring cohomology theory. Given an integration $f \mapsto f_!$ on $A$, there is a unique integration with supports on $A$ extending $f$.
\end{theorem}

The proof of the theorem follows the usual strategy of decomposing a given projective morphism $f$ into a closed embedding and a projection from a trivial projective bundle. For the closed embedding case, one employs the deformation to the normal cone. For the projective bundle case, there is a projective bundle formula (see Theorem 3.9 in \cite{panin2003oriented}). Moreover, Levine proves the following proposition which will be useful (Theorem 1.24(a) in \cite{levine2008oriented}).

\begin{proposition}\label{integrationDecompose} (Levine) Let $f: X \rightarrow Y$ be a projective morphism. Factor $f$ as a closed embedding $i: X \rightarrow \Ps^n \times Y$ where $i = i_X \times f$ and where $i_X$ is some choice of closed embedding for $X \hookrightarrow \Ps^n$, and a projection $\pi: \Ps^n \times Y \rightarrow Y$. Then $f = \pi\circ i$, and the morphism $f_! = \pi_! \circ i_!$ does not depend on the choice of factorization.
\end{proposition}


\subsubsection{Generalized Todd classes}

In this section, we will discuss Smirnov's generalized Riemann--Roch theorem for ring cohomology theories and integrations with support. We will follow \cite{smirnov2007riemann}.

We present the definition of an operation for a ring cohomology theory; it is a ring homomorphism. We avoid the categorical generalities of Smirnov's definition in \cite{smirnov2007riemann}, and only present what is necessary for our discussion.
\begin{definition}
    An \textbf{operation} $\phi$ of a ring cohomology $A$ theory is a natural transformation $A \rightarrow A$ such that the following diagram commutes:        
        \begin{center}
        \begin{tikzcd}
        A(X) \otimes A(Y) \arrow[r, "\phi \otimes \phi"] \arrow[d, "\times"]
        & A(X) \otimes A(Y) \arrow[d, "\times"] \\
        A(X \times Y) \arrow[r, "\phi"]
        & A(X \times Y).
        \end{tikzcd}
        \end{center}

    Moreover, we require that it is compatible with the cup product structure on cohomology with supports; namely, that it is also defined for cohomology with support:
        \begin{center}
        \begin{tikzcd}
        A_Z(X) \otimes A_{Z'}(X) \arrow[r, "\phi \otimes \phi"] \arrow[d, "\smile"]
        & A_{Z}(X) \otimes A_{Z'}(X) \arrow[d, "\smile"] \\
        A_{Z\cap Z'}(X) \arrow[r, "\phi"]
        & A_{Z\cap Z'}(X).
        \end{tikzcd}
        \end{center}
\end{definition}

In general, one can consider two different ring cohomology theories $A$ and $B$, and one needs a functor to compare the monoidal structures on two possibly different output categories; for example, the Chern character is an operation between $K$-theory and cohomology. However, we will only consider ring operations $\phi: A \rightarrow A$ in this paper.

\begin{example}
    The total Steenrod power $\phi = Q^\bullet := \sum_s Q^s$ given by the formal series of simplicial Steenrod operations constructed by \cite{joshua2001motivic} is an operation on motivic cohomology by the Cartan formula.
\end{example}

\begin{example}
    The total Steenrod power $\phi = P^\bullet := \sum_s P^s$ given by the formal series of Steenrod operations in the sense of Voevodsky \cite{voevodsky2003reduced} is also an operation on motivic cohomology, again since it satisfies the Cartan formula.
\end{example}

Before we continue, we must consider the notion of Chern classes for our generalized ring cohomology theories. See Definition 1.12 in \cite{panin2009oriented}.

\begin{definition}\label{definitionofChernstructure}
    A \textbf{Chern structure} on a ring cohomology theory $A$ is an assignment $L \mapsto c(L)$ for every smooth variety $X$ and every line bundle $L$ over $X$ an even degree element $c(L) \in A(X)$ satisfying the following properties:

    \begin{enumerate}
        \item (Functoriality) $c(L_1) = c(L_2)$ if $L_1 \cong L_2$, and for every map $f: Y \rightarrow X$ between smooth varieties, we have $f^*c(L) = c(f^*L)$;
        \item (Non-degeneracy) the map $A(X) \oplus A(X) \rightarrow A(X\times \mathbb{P}^1)$ given by sending $(x,y)$ to $\pi^*x + \pi^*y\smile \xi$ is an isomorphism, where $\xi = c(\mathcal{O}(-1))$ for the tautological bundle $\mathcal{O}(-1)$ over $\mathbb{P}^1$, and $\pi: X \times \mathbb{P}^1 \rightarrow X$ is the projection; 
        \item (Vanishing) $c(X \times \A^1) = 0 \in A(X)$ for the trivial line bundle $X \times \A^1$ on any smooth variety $X$.
    \end{enumerate}

    The element $c(L) \in A(X)$ is called the \textbf{Chern class} of $L$.
    
\end{definition}

Panin and Smirnov show that equipping a ring cohomology theory with an integration is equivalent to equipping it with a Chern structure; in fact, the set of integrations is in bijection with the set of Chern structures (Theorem 2.5 in \cite{panin2009oriented}). One obtains higher Chern classes using the projective bundle formula, which holds for our notion of ring cohomology theory (see Theorem 3.9 in \cite{panin2003oriented}).

Given an operation for a ring cohomology theory equipped with Chern classes, there is a notion of an inverse Todd class. As in the classical setting, one uses the splitting principle, which is available to us in our general setting via the projective bundle formula (see Definition 5.1.1 from the preprint \cite{panin2000oriented}).

First we define $\mathbb{P}^\infty$ as the direct limit of projective spaces $\mathbb{P}^\infty := \varinjlim \mathbb{P}^k$ induced by inclusions of all finite dimensional subspaces. Similarly, we define $A(P^\infty) := \varprojlim A(\mathbb{P}^k)$ where $A$ is our ring cohomology theory; we extend our given operation $\phi$ to $A(\mathbb{P}^\infty)$ by taking the direct limit of natural operations $\phi$ on the projective spaces in the directed system. It follows from the projective bundle theorem that we can identify $A(\mathbb{P}^\infty)$ with the ring $A(\textnormal{Spec }k)[[u]]$ of formal power series with coefficients in $A(\textnormal{Spec }k)$, with the class $u$ pulling back along any of the inclusions $\mathbb{P}^k \rightarrow \mathbb{P}^\infty$ to $c_1(\mathcal{O}(1))$. We then define the inverse Todd genus for $\phi$ on $\mathbb{P}^{\infty}$, then extend the definition to all vector bundles by splitting principle.

\begin{lemma}
    The operation $\phi$ on $\mathbb{P}^\infty$ maps the generator $u \in A(\mathbb{P}^\infty)$ to a formal power series $\phi(u)$ which is divisible by $u$.
\end{lemma}
\begin{proof}
    The constant term in the formal power series $\phi(u)$ is given by the action of $\phi$ on the class $u$ pulled back to $A(\textnormal{Spec }k)$ via compatible choices of inclusion $\textnormal{Spec }k \rightarrow \mathbb{P}^n$ in the direct system. However, the class $u$ pulled back to $\textnormal{Spec }k$ is $0$, since $u = c_1(\mathcal{O}(1))$ and $\mathcal{O}(1)$ trivializes over a point.
\end{proof}

\begin{definition} We define the \textbf{inverse Todd genus associated to an operation} $\phi$ on a ring cohomology theory $A$ equipped with a Chern structure, for any vector bundle $E$ on a smooth variety $X$ as follows.
    \begin{enumerate}
    \item We first define the \textbf{inverse Todd genus} of $\phi$ as the formal power series
    \begin{equation*}
        itd_{\phi}(u) = \frac{\phi(u)}{u} \in A[[u]]
    \end{equation*}
    where we consider the ring operation $\phi: A(\mathbb{P}^\infty) \rightarrow A(\mathbb{P}^\infty).$

    \item Let $t_1, ..., t_n$ be independent variables and let $\sigma_1, ..., \sigma_n$ be the elementary symmetric polynomials in these variables (i.e., $\sigma_1 = t_1 + ... + t_n, \sigma_2 = t_1t_2 + ... + t_{n-1}t_n, ..., \sigma_n = t_1...t_n$). We define
    \begin{equation*}
        itd_{\phi}(\sigma_1, ..., \sigma_n) = \prod_{i=1}^n itd_{\phi}(t_i) \in A[[\sigma_1, ..., \sigma_n]].
    \end{equation*}

    \item We define the \textbf{inverse Todd genus} of a vector bundle $E$ with rank $n$ over a smooth variety $X$ as the evaluation of the series $itd_{\phi}(\sigma_1, ..., \sigma_n)$ on the Chern classes of $E$:
    \begin{equation*}
        itd_{\phi}(E) = itd_{\phi}(c_1(E),..., c_n(E)) \in A(X).
    \end{equation*}
    
    \end{enumerate}
\end{definition}

The Todd genus will measure the failure for the given operation $\phi$ to commute with the proper pushforward (which is uniquely determined by the Chern structure). However, we require one additional hypothesis on $\phi$ to make the notion of Todd genus precise.

\begin{definition}\label{toddgenusdefinition}
    Assume that the series $itd_\phi(u)$ is invertible in $A[[u]]$ (this is equivalent to the free term of the series being a unit in $A$). In this case, we say that $\phi$ has \textbf{well-defined Todd genus}.
    \begin{enumerate}
        \item Define the \textbf{Todd genus} of $\phi$ as the multiplicative inverse of $itd_\phi(u)$
        \begin{equation*}
            td_\phi(u) := \frac{u}{\phi(u)} \in A[[u]].
        \end{equation*}
        \item Set $td_\phi(\sigma_1,...,\sigma_n) = \prod_{i=1}^n td_\phi(t_i) \in A[[\sigma_1,...,\sigma_n]]$.
        \item Define the \textbf{Todd genus} of a vector bundle $E$ with rank $n$ over a smooth variety $X$ as 
        \begin{equation*}
            td_\phi(E) = td_\phi(c_1(E), ..., c_n(E)) \in A(X).
        \end{equation*}
    \end{enumerate}
\end{definition}

As an example, we compute the inverse Todd genus for the simplicial operations $Q^\bullet$ on motivic cohomology. We first compute the motivic cohomology of $\mathbb{P}^\infty$. By the projective bundle formula, this only depends on the Chow groups of projective space, and the motivic cohomology of our underlying field. We present below the main theorem of Geisser--Levine in \cite{geisser2000k}.

\begin{theorem}
    (Geisser--Levine) The mod $\ell = p$ motivic cohomology of a field $k$ of characteristic $p$ is zero in bidegrees $(i,j)$ where $i \neq j$. For $i = j$, it is given by the $i$-th Milnor $K$-theory group $K^M_i(k)/p$. In particular, if $k$ is algebraically closed, then the mod $\ell = p$ motivic cohomology groups vanish for all $(i,j)$ except in bidegree $(0,0)$.
\end{theorem}
\begin{proof}
    We only prove the last statement: by definition, $K^M_i(k)$ is the $i$-fold tensor product of $K^M_1(k)$ with itself, modulo Steinberg relations (i.e., Milnor $K$-theory is generated as a ring in degree $1$). We also have that $K^M_1(k) \cong k^*$ as groups; since $k$ is algebraically closed, $K^M_1(k)/p = 0$, since every element of $k^*$ admits a $p$-th root. Thus all the higher Milnor $K$-groups vanish. In degree $0$, we see that $K^M_0(k) \cong \Z$. 
\end{proof}

On the other hand, if $\ell \neq p$, we have by the Beilinson--Lichtenbaum conjecture that the mod $\ell$ motivic cohomology groups are isomorphic to the \etale cohomology groups $H^i(k_{\textnormal{\'et}},\mu^{\otimes j}_\ell)$. Since $k$ is algebraically closed, we see that $\mu_\ell$ identifies with the constant sheaf $\underline{\Z/\ell}$, and $\mu_\ell^{\otimes j} \cong \underline{\Z/\ell}$ for all $j$. Moreover, \etale cohomology of a field agrees with the group cohomology of its absolute Galois group, so since $k$ is algebraically closed, we see that the \etale cohomology groups all vanish above degree $0$. In degree $0$, we see that $H^0(k_{\textnormal{\'et}},\mu^{\otimes j}_{\ell}) \cong \Z/\ell$ for any $j \geq 0$.

\begin{proposition}
    The groups $H^i(\mathbb{P}^n,\Z/\ell(j))$ are $\Z/\ell$ in degrees $(i,j) = (2m,j)$ where $j,m \geq 0$, for $\ell \neq p$. The groups are $\Z/p$ in degrees $(i,j) = (2m,m)$ for $m \geq 0$, for $\ell = p$.
\end{proposition}
\begin{proof}
    By the projective bundle theorem, for a given projective space $\mathbb{P}^n$, we see that the mod $\ell$ motivic cohomology is given by the ordinary Chow groups of projective space tensored with the mod $\ell$ motivic cohomology groups of the underlying field. That is, we have
    \begin{equation*}
        H^i(\mathbb{P}^n,\Z/\ell(j)) \cong \bigoplus^n_{m=0} H^{i-2m}(\textnormal{Spec }k, \Z/\ell(j-m))
    \end{equation*}
    where $k$ is our underlying field. A given element in degree $(i,j)$ is then a series $\alpha = \sum_{m=0}^n \alpha_m u^m$, where $\alpha_j \in H^{i-2m}(\textnormal{Spec }k, \Z/\ell(j-m))$, and $u \in H^2(\mathbb{P}^n,\Z/\ell(1))$ is the generator for the normal Chow groups of projective space.

    When $\ell \neq p$, by Beilinson--Lichtenbaum, we see that the only vanishing groups for our field is when $i = 2m$, for arbitrary $j$. In this case, the group is $\Z/\ell$. When $\ell = p$, all the groups vanish except when $(i,j) = (2m,m)$. That is, the only non-zero higher Chow groups mod $p$ of projective space are those given by the ordinary Chow groups mod $p$, over an algebraically closed field of characteristic $p$.
\end{proof}

\begin{lemma}
    The mod $\ell \neq p$ operation $Q^\bullet$ sends the generator $u \in H^2(\mathbb{P}^\infty,\Z/\ell(1))$ to $u + u^\ell$.
\end{lemma}
\begin{proof}
    The class $u$ is in $H^2(\mathbb{P}^\infty,\Z/\ell(1))$, and so $Q^s(u) = 0$ if $2s > 2$ by vanishing of the Steenrod operations on classes of lower degree. Thus, $Q^\bullet(u) = Q^0(u) + Q^1(u)$. we see that $Q^1(u) = u^\ell$ by definition of the Steenrod operations. By Proposition 5.5 in \cite{joshua2001motivic}, using naturality of the operations $Q^\bullet$ on any mod $\ell$ sheaf cohomology, it suffices to check $Q^0$ on the cohomology of the constant sheaf $\underline{\Z/\ell}$. However, we see that $Q^0 = id$ in this case, since $x^\ell = x \in \Z/\ell$.
\end{proof}

\begin{proposition}\label{inverseToddGenusofSimplicialOperations}
    The inverse Todd genus of $Q^\bullet$ of a vector bundle $E$ with rank $n$ over a smooth variety $X$ is given by
    \begin{equation*}
        itd_{Q^\bullet}(E) = c_n(E)^{p-1}
    \end{equation*}
    where $Q^\bullet$ acts on mod $p$ motivic cohomology of $X$.
\end{proposition}
\begin{proof}
If we apply $Q^\bullet$ to the generator $u \in H^2(\mathbb{P}^\infty,\Z/\ell(1))$, we obtain $Q^\bullet(u) = u^p$. The proof is as follows: again by vanishing of Steenrod operations on classes of lower degree, we see that $Q^s(u) = 0$ if $2s > 2$. Thus, $Q^\bullet(u) = Q^0(u) + Q^1(u)$. However, we also have $Q^0(u) \in H^2(P^\infty,\Z/p(p)) = 0$ as we computed above. Moreover, $Q^1(u) = u^p$ by definition of the simplicial operations. Tracing through the definition of an inverse Todd genus for a bundle, we obtain the result.
\end{proof}

\remark{} The result of Geisser--Levine holds for \textit{any} field of characteristic $p$; thus, the vanishing of groups $H^2(\mathbb{P}^\infty, \Z/p(p))$ holds for any field of characteristic $p$, not necessarily algebraically closed, for $p > 2$. If $p = 2$, and the field is not algebraically closed, it is possible for $H^2(\mathbb{P}^\infty,\Z/2(2))$ to be non-zero. From the projective bundle theorem, we see that an element $\alpha$ is given by $\sum \alpha_m u^m$ for $\alpha_m \in H^{2-2m}(\textnormal{Spec }k, \Z/2(2-m))$. However, again by Geisser--Levine, the only a priori non-zero term is the constant term $\alpha_0 \in K^M_2(k)/2$. Since the series $Q^\bullet(u)$ must be divisible by $u$, we see that $Q^0(u)$ still vanishes even in this case. We conclude this section with a computation of the inverse Todd genus for the mod $\ell \neq p$ simplicial operations and the mod $\ell$ motivic operations.

\begin{lemma}\label{inverseToddgenusforModL}
The inverse Todd genus of $Q^\bullet$ acting on mod $\ell \neq p$ motivic cohomology of a vector bundle $E$ with rank $n$ over a smooth variety $X$ is given by
\begin{equation*}
    itd_{Q^\bullet}(E) = \prod_{i=1}^n (1+\alpha_i)^{\ell-1}
\end{equation*}
where $\alpha_i$ is the $i$-th Chern root of the bundle $E$. Moreover, the inverse Todd genus for the motivic operation $P^\bullet$ mod $\ell$ (including when $\ell = p$) is given by the same formula.
\end{lemma}
\begin{proof}
    The operations in the lemma all map the generator $u \in H^2(\mathbb{P}^\infty,\Z/\ell(1))$ to the formal power series $u + u^{\ell}$. Thus the formal power series $\frac{\phi(u)}{u}$ is given by $1 + u^{\ell-1}$. One then traces through the definition of the inverse Todd genus and obtains the result.
\end{proof}

\subsection{Riemann--Roch formulas for projective pushforward}\label{RiemannRochTypeFormulasSection}
We now present two theorems of Smirnov from \cite{smirnov2007riemann}, which are similar to the main theorems of Panin--Smirnov in \cite{panin2000oriented}, \cite{panin2004riemann}. The main theorems are Grothendieck--Riemann--Roch style formulas describing failure of commutativity of a given ring operation $\phi$ with proper pushforward on a ring cohomology theory $A$, with the defects arising from the Todd classes. We present the theorems from \cite{smirnov2007riemann} because Smirnov specifically deals with the case of cohomology with support.

\begin{theorem}\label{smirnovEmbeddingTheorem} (Relative Wu formula) Let $X \xhookrightarrow{i} P$ and $X' \xhookrightarrow{i'} P'$ be closed immersions into smooth varieties $P$ and $P'$, and suppose we have closed immersions $f$ and $f'$ such that the following diagram commutes:
\begin{center}
\begin{tikzcd}
P \arrow[r, "f"]
& P' \\
X \arrow[r, "f'"] \arrow[hookrightarrow, u,"i"]
& X' \arrow[hookrightarrow, u, "i'"].
\end{tikzcd}
\end{center}
Suppose we have an operation $\phi$ on a ring cohomology theory $A$ equipped with a transfer (or equivalently a Chern structure) so that the proper pushforward $f_!$ on $A$ is defined. Then we have
\begin{equation*}
    \phi(f_!(\alpha)) = f_!(\phi(\alpha) \smile itd_\phi(N_PP'))
\end{equation*}
where $N_PP'$ denotes the normal bundle of $P$ in $P'$ and $\alpha \in A_X(P)$.
\end{theorem}

Smirnov actually only requires the embedding to be proper on the support $X$ for the theorem to hold (i.e., $f$ need not be a closed immersion, but $f'$ must be). This allows us to deal with quasi-projective varieties in general; however, we will only work with ambient spaces given by projective varieties.

Recall that motivic cohomology is isomorphic to higher Chow groups for smooth varieties. If $f: X \rightarrow Y$ is a proper map, then the map $f \times Id: X \times  \Delta^i \rightarrow Y \times  \Delta^i$ is proper. This induces a pushforward map on algebraic cycles $(f \times Id)_*: z_n(X,i) \rightarrow z_n(Y,i)$ given by sending a closed subscheme $W$ to $[k(W):k((f\times Id)(W))](f \times id)(W)$ if the dimension of $(f\times Id)(W)$ equals the dimension of $W$, and $0$ otherwise. This induces a map $f_!$ on higher Chow groups.

\begin{corollary}\label{WuformulaforSimplicialSteenrodOps}
    Let $\phi = Q^\bullet$ be defined mod $\ell = p$, let $A$ be given by motivic cohomology, and let $f_!$ be the proper pushforward for a diagram as in the hypothesis of Theorem \ref{smirnovEmbeddingTheorem}. Then we have
    \begin{equation*}
        Q^\bullet(f_!(\alpha)) = f_!(\phi(\alpha) \smile c_{\textnormal{top}}^{p-1}(N_PP'))
    \end{equation*}
    where $c_{\textnormal{top}}(N_PP')$ is the top Chern class of the normal bundle $N_PP'$.
\end{corollary}

Theorem \ref{smirnovEmbeddingTheorem} deals with the closed embedding case only; Theorem \ref{smirnovFormuoliTheorem} generalizes to any proper map. However, Theorem \ref{smirnovFormuoliTheorem} requires an additional hypothesis on the operation $\phi$; namely, $\phi$ must have well-defined Todd genus.

\begin{theorem}\label{smirnovFormuoliTheorem}
Let $X \xhookrightarrow{i} P$ and $X' \xhookrightarrow{i'} P'$ be closed immersions into smooth varieties $P$ and $P'$, and suppose we have proper maps $f$ and $f'$ such that the following diagram commutes:
\begin{center}
\begin{tikzcd}
P \arrow[r, "f"]
& P' \\
X \arrow[r, "f'"] \arrow[hookrightarrow, u,"i"]
& X' \arrow[hookrightarrow, u, "i'"].
\end{tikzcd}
\end{center}
Suppose we have an operation $\phi$ on a ring cohomology theory $A$ equipped with a transfer (or equivalently a Chern structure) so that the proper pushforward $f_!$ on $A$ is defined. Suppose that $\phi$ has well-defined Todd genus. Then we have
\begin{equation*}
    \phi(f_!(\alpha)) \smile td_\phi(TP') = f_!(\phi(\alpha) \smile td_\phi(TP))
\end{equation*}
where $TP'$ and $TP$ denote the tangent bundles of $P'$ and $P$ respectively, and $\alpha \in A_X(P)$. We call this equation a \textbf{Riemann--Roch type formula}.
\end{theorem}

Theorems \ref{smirnovEmbeddingTheorem} and \ref{smirnovFormuoliTheorem} again hold when one has two different cohomology theories with an operation $\phi$ that is able to compare the two appropriately, such as $K$-theory and singular cohomology, with the Chern character. In this classical setting, the Todd genus is the usual Todd class, and one recovers the Grothendieck--Riemann--Roch formula. We also immediately obtain the following corollary from Lemma \ref{inverseToddgenusforModL}.

\begin{corollary}\label{GRRforModL}
    Let $\phi = Q^{\bullet}$ be defined for mod $\ell \neq p$, or $\phi = P^\bullet$ for any prime $\ell$. Let $A$ be motivic cohomology and let $f_!$ be the proper pushforward for a diagram as in the hypothesis of Theorem \ref{smirnovFormuoliTheorem}. Then we have
\begin{equation*}
    \phi(f_!(\alpha)) \smile td_\phi(TP') = f_!(\phi(\alpha) \smile td_\phi(TP))
\end{equation*}
where $TP'$ and $TP$ denote the tangent bundles of $P'$ and $P$ respectively, and $\alpha \in H_X^*(P,\Z/\ell(*))$.
\end{corollary}

The proof of Theorem \ref{smirnovFormuoliTheorem} is in the same spirit as Satz 3.2 of Atiyah--Hirzebruch \cite{atiyah1961cohomologie} and other Riemann--Roch type formulae where one first realizes the given projective morphism as the composition of a closed embedding and a projection, then proves the formula for each case separately. The case of a closed embedding is exactly Theorem \ref{smirnovEmbeddingTheorem}. However, one needs to take care formalizing the argument for cohomology with support on a singular subvariety; to do this, we employ the formalism of Panin--Smirnov \cite{panin2009oriented}, \cite{panin2004riemann} and Levine's extension of their theory to the case of supports (Section 1 of \cite{levine2008oriented}). 


\subsubsection{Closed embeddings}
We analyze the proof of Theorem \ref{smirnovEmbeddingTheorem} for our operations $Q^\bullet$ mod $\ell$, for any prime $\ell$. First we recall the deformation to the normal cone (see Chapter 5 in \cite{fulton2013intersection}).

Let $X$ be a closed subvariety of $Y$. Then there is a family of embeddings $X \hookrightarrow Y_t$ parameterized by $t \in \mathbb{P}^1$ (resp. $t \in \mathbb{A}^1$), such that for $t \neq \infty$ (resp. $t \neq 0$), the embedding is the given closed immersion $X \hookrightarrow Y$, and for $t = \infty$ (resp. $t = 0$), the embedding is the zero section of $X$ in the normal cone $N_XY$. Moreover, if $X$ is smooth, then the normal cone is actually a normal bundle. The family of embeddings is constructed by first building a flat family $M = M_XY$, with a closed embedding of $X \times \mathbb{P}^1$ in $M$ and a flat morphism $\rho: M \rightarrow \mathbb{P}^1$ so that the following diagram commutes:
\begin{center}
        \begin{tikzcd}
            X \times \mathbb{P}^1 \arrow[hook,rr] \arrow[dr,"\pi"] & & M \arrow[dl,"\rho"] \\
            & \mathbb{P}^1 &
        \end{tikzcd}
    \end{center}
where $\pi$ is the projection, and such that
\begin{enumerate}
\item over $\mathbb{P}^1 - \{\infty\} = \mathbb{A}^1$, we have $\rho^{-1}(\mathbb{A}^1) = Y \times \mathbb{A}^1$ and the embedding is the product embedding $X\times \mathbb{A}^1 \hookrightarrow Y \times \mathbb{A}^1$,
\item over $\infty$, the divisor $M_{\infty} = \rho^{-1}(\infty)$ is the sum of two effective Cartier divisors
\begin{equation*}
M_\infty = \mathbb{P}(N_XY\oplus 1)  +  Bl_XY,
\end{equation*}
where $\mathbb{P}(N_XY \oplus 1)$ is the hyperplane at infinity for the normal cone, and $Bl_XY$ is the blowup of $Y$ along $X$. We call the flat family $M$ the \textbf{deformation to the normal cone}.
\end{enumerate}

If we wish to parameterize with $t \in \mathbb{A}^1$ (as in 2.2.7 in \cite{panin2003oriented}), we obtain the same diagram
\begin{center}
        \begin{tikzcd}
            X \times \mathbb{A}^1 \arrow[hook,rr] \arrow[dr,"\pi"] & & M \arrow[dl,"\rho"] \\
            & \mathbb{A}^1 &
        \end{tikzcd}
    \end{center}
where the fiber of $\rho$ over $1 \in \mathbb{A}^1$ is canonically isomorphic to $Y$, and the embedding $X \times \mathbb{A}^1 \hookrightarrow M$ is the initial embedding $X \hookrightarrow Y$ over $1$, and the fiber of $\rho$ over $0$ is canonically isomorphic to $N_XY$, with the embedding over $0$ given by the zero section $X \hookrightarrow N_XY$. Thus we obtain maps of pairs
\begin{equation*}
    (N_XY, N_XY - X) \xrightarrow{i_0} (M, M- (X\times \mathbb{A}^1)) \xleftarrow{i_1} (Y, Y-X)
\end{equation*}
allowing us to relate the cohomology of the Thom space of the normal bundle, with the cohomology of the pair given by the ambient space modulo the complement of the subspace. In fact, for a ring cohomology theory, we have the following general theorem (Theorem 2.2 in \cite{panin2003oriented}, Lemma 1.14 in \cite{levine2008oriented}).

\begin{theorem}\label{normalBundleAmbientPairDeformationIsomorphism}
    For a ring cohomology theory $A$, the induced maps of pairs
    \begin{equation*}
        A_X(N_XY) \xleftarrow{i_0^*} A_{X \times \mathbb{A}^1}(M) \xrightarrow{i_1^*} A_X(Y)
    \end{equation*}
    are isomorphisms. Moreover, for each closed subset $Z \subset X$, the induced maps of pairs
    \begin{equation*}
        A_Z(N_XY) \xleftarrow{i_0^*} A_{Z \times \mathbb{A}^1}(M) \xrightarrow{i_1^*} A_Z(Y)
    \end{equation*}
    are isomorphisms.
\end{theorem}


It is now appropriate for us to define a \textbf{Thom structure} in the sense of Panin--Smirnov \cite{panin2000oriented} (Definition 3.3 in \cite{panin2003oriented}). It is essentially the notion of a Thom isomorphism.

\begin{definition}\label{definitionofThomstructure}
    A \textbf{Thom structure} on a ring cohomology theory $A$ is an assignment $L \mapsto \tau(L)$ for every smooth variety $X$ and every line bundle $L$ over $X$ an even degree element $\tau(L) \in A_X(L)$ satisfying the following properties:

    \begin{enumerate}
        \item (Functoriality) $\phi^*(\tau(L_1)) = \tau(L_2)$ for every isomorphism $\phi: L_1 \cong L_2$ of line bundles; moreover $f_L^*(\tau(L)) = \tau_{L_Y}$ for every morphism $f: Y \rightarrow X$ and each line bundle $L$ over $X$, where we write $L_Y = L \times_X Y$ as the pullback line bundle over $Y$ and $f_L: L_Y \rightarrow L$ is the projection;
        \item (Non-degeneracy) the map $-\smile \tau(1): A(X) \rightarrow A_X(X\times \mathbb{A}^1)$ is an isomorphism, where $X$ is viewed as the zero section of the trivial line bundle.
    \end{enumerate}
    The element $\tau(L) \in A_X(L)$ is called the \textbf{Thom class} of $L$.
    
\end{definition}

Recall we have a notion of Thom class for higher Chow groups from localization results. Moreover, by splitting principle, one extends the notion of Thom class for line bundles to general vector bundles. The main results of Panin--Smirnov \cite{panin2000oriented}, \cite{panin2003oriented}, \cite{panin2009oriented} are then that the notions of Chern structures, Thom structures, and proper pushforwards are all in bijection with each other for a given ring cohomology theory; that is, given any one of the structures, one can recover the other two uniquely. This extends to the case of ring cohomology with support by work of Levine in \cite{levine2008oriented}; one such result is as follows (see Theorem 1.12 and Proposition 1.15 in \cite{levine2008oriented}).

\begin{proposition}\label{embeddingPushforwardIsGivenByThomclass}
    Let $X \xhookrightarrow{i} Y$ be a closed immersion of smooth varieties and consider the maps
        \begin{equation*}
        A_X(N_XY) \xleftarrow{i_0^*} A_{X \times \mathbb{A}^1}(M) \xrightarrow{i_1^*} A(Y)
    \end{equation*}
    where $i_1^*$ is induced from the map of pairs $(Y, \varnothing) \xrightarrow{i_1} (M, M - (X\times \mathbb{A}^1))$. In this case, $i_0^*$ is still an isomorphism (though $i_1^*$ may not be). Moreover, the proper pushforward $i_!: A(X) \rightarrow A(Y)$ agrees with the composition 
    \begin{equation*}
        A(X) \xrightarrow{\Phi} A_X(N_XY) \xrightarrow{i_1^*\circ (i_0^*)^{-1}} A(Y)
    \end{equation*}
    where the map $\Phi$ is given by $A(X) \xrightarrow{\pi^*} A(N_XY) \xrightarrow{\smile \tau} A_X(N_XY)$, where $\tau$ is the Thom class of the normal bundle $N_XY \xrightarrow{\pi} X$, and $X$ is thought of as the zero section in $N_XY$. More generally, for any closed subset $Z \subset X$ and any closed subset $W \subset Y$ that contains $i(Z)$, we have the map 
    \begin{equation*}
        A_Z(N_XY) \xleftarrow{i_0^*} A_{Z \times \mathbb{A}^1}(M) \xrightarrow{i_1^*} A_W(Y)
    \end{equation*}
    where again $i_0^*$ is an isomorphism. In this case, the proper pushforward on cohomology with support $i_!: A_Z(X) \rightarrow A_W(Y)$ is given by the composition
    \begin{equation*}
        A_Z(X) \xrightarrow{\Phi} A_Z(N_XY) \xrightarrow{i_1^*\circ (i_0^*)^{-1}} A_W(Y)
    \end{equation*}
    where the map $\Phi$ is given by $A_Z(X) \xrightarrow{\pi^*} A_{\pi^{-1}(Z)}(N_XY) \xrightarrow{\smile\tau} A_Z(N_XY)$, where $\tau \in A_X(N_XY)$ is the Thom class of the normal bundle $N_XY$.
    
\end{proposition}

In other words, the proper pushforward for a closed embedding is exactly given by cupping with the Thom class (aka the Thom isomorphism), up to deformation of the normal cone. In the original context, Proposition \ref{embeddingPushforwardIsGivenByThomclass} is actually proven to show that, given a transfer structure, one can obtain a unique transfer with supports; Levine then constructs the transfer with supports explicitly. 


Let us now verify Corollary \ref{WuformulaforSimplicialSteenrodOps}, the relative Wu formula for our simplicial operations $Q^\bullet$ mod $p$.
\begin{proof}
    Let $X \xhookrightarrow{i} Y$ be a closed immersion. Then we have
    \begin{equation*}
        Q^\bullet(i_!\alpha) = Q^\bullet(\varphi\circ (\pi^*\alpha \smile \tau))
    \end{equation*}
    where $\varphi = i^*_1\circ (i_0^*)^{-1}$ commutes with $Q^\bullet$ by naturality of the operations, and $\tau$ is the Thom class of the normal bundle $N_XY$. After commuting $\varphi$ and $Q^\bullet$, we obtain
    \begin{equation*}
        \varphi(Q^\bullet(\pi^*\alpha \smile \tau)) = \varphi( \pi^*Q^\bullet\alpha \smile Q^\bullet\tau)
    \end{equation*}
    but now we note that $Q^\bullet\tau = \tau \smile \pi^*(itd_{Q^\bullet}(N_XY))$ by definition of the inverse Todd genus, the splitting principle, and tracing through the equivalences between Thom structures and Chern structures. This leaves us with $\varphi(\pi^*Q^\bullet\alpha \smile \pi^*(itd_{Q^\bullet}(N_XY)) \smile \tau) = \varphi(\pi^*(Q^\bullet\alpha\smile itd_{Q^\bullet}(N_XY)) \smile \tau) = i_!(Q^\bullet\alpha \smile itd_{Q^\bullet}(N_XY))$. We have computed the inverse Todd genus in Proposition \ref{inverseToddGenusofSimplicialOperations}, which finishes the proof.
\end{proof}

\subsubsection{Projective bundles}\label{projectivebundlesection}

As mentioned, the strategy of the proof of Theorem \ref{smirnovFormuoliTheorem} involves decomposing a given proper map $f$ into a composition $f = \pi \circ i$ where $i$ is a closed embedding and $\pi$ is a projection from a projective bundle. More specifically, if $f: X \rightarrow Y$ is a proper map of smooth varieties, then we can decompose it as
\begin{center}
        \begin{tikzcd}
            & Y \times \mathbb{P}^n \arrow[dr, "\pi"] & \\
            X \arrow[rr,"f"] \arrow[hook,ur,"f \times e"] & & Y 
        \end{tikzcd}
    \end{center}
where the map $e$ is some choice of embedding of $X \hookrightarrow \mathbb{P}^n$ (recall all varieties are quasiprojective; since the map $f$ is proper, the map $i = f \times e$ will be a closed embedding). Recall also that the choice of decomposition does not matter (Proposition \ref{integrationDecompose}). The resulting diagram on cohomology with supports is then
\begin{center}
        \begin{tikzcd}
            & A_{Z' \times \mathbb{P}^n}(Y \times \mathbb{P}^n) \arrow[dr, "\pi_!"] & \\
            A_Z(X) \arrow[rr,"f_!"] \arrow[hook,ur,"(f \times e)_!"] & & A_{Z'}(Y) 
        \end{tikzcd}
    \end{center}
where the maps are proper on the closed supports $Z \subset X$ and send $Z$ into $Z' \subset Y$. The proof of Theorem \ref{smirnovFormuoliTheorem} then involves proving the formula for the closed embedding case and projective bundle case separately. We have already dealt with the embedding case, and this section will cover the projective bundle case.

Suppose $X \subset P$ is a closed subvariety with $P$ smooth. We have the following projective bundle theorem for the case of supports (Corollary 3.17 in \cite{panin2003oriented}, Remark 1.17 in \cite{levine2008oriented}), by using the exact sequence for the pairs $(P,X)$ and $(P \times \mathbb{P}^n, X \times \mathbb{P}^n)$.

\begin{theorem}
    Let $X \xhookrightarrow{i} P$ be a closed immersion with $P$ smooth. For any closed subvariety $Y \subset P$ that also contains $X$, we have an isomorphism of $A_Y(P)$-modules
    \begin{equation*}
        A_{X \times \mathbb{P}^n}(P \times \mathbb{P}^n) \cong A_X(P) \otimes_{A(P)} A(P)[t]/(t^{n+1})
    \end{equation*}
    with $a \otimes t^i$ mapping to $\pi^*(a) \smile c_1(\mathcal{O}(1))^i$, where $\pi: P \times \mathbb{P}^n \rightarrow P$ is the projection.
\end{theorem}

It suffices now to compute the proper pushforward $\pi_!$ induced by the projective bundle. For more general ring cohomology theories $A$, the situation is more complicated: one passes from cohomology of the projective bundle $A(P \times \mathbb{P}^n)$ to the inverse limit $A(P \times \mathbb{P}^\infty) := \varprojlim A(P \times \mathbb{P}^n)$ and one reduces the proof to proving the formula for the infinite projective bundle. For a ring cohomology theory $A$ equipped with a Chern structure, the fact that isomorphism classes of line bundles is a group yields a formal group law associated to the Chern structure; i.e., $c_1(L_1 \otimes L_2) = F_A(c_1(L_1),c_1(L_2))$ where $F_A$ is a formal group law over $A(\textnormal{Spec } k)$. The proper pushforward $\pi_!$, and by extension, the proof of the Riemann--Roch formula for the infinite projective bundle, is then determined by this formal group law.

Fortunately, we will now return to our original setting where our ring cohomology theory $A$ is given by higher Chow groups, equipped with the usual notion of Chern class. In this case, we have from normal Chow group theory that $c_1(L_1 \otimes L_2) = c_1(L_1) + c_1(L_2)$, and by definition the proper pushforward $\pi_!: A_X(P) \otimes_{A(P)} A(P)[t]/(t^{n+1}) \rightarrow A_X(P)$ is given by 
\begin{equation*}
    \pi_!(a\otimes t^i) = 
        \begin{cases}
        0 & \text{if $i \neq n$}\\
        a & \text{if $i = n$}
    \end{cases}
\end{equation*}
for an element $a \otimes t^i \in A_X(P) \otimes_{A(P)} A(P)[t]/(t^{n+1})$. Geometrically, the only subvarieties we are projecting are those originally in $P$, since the class $a \otimes t^n$ corresponds to the product of a given subvariety with the class of a point. We recall Corollary \ref{GRRforModL}.

\begin{corollary}
    For the mod $\ell \neq p$ operations $Q^\bullet$ on mod $\ell$ motivic cohomology, and a diagram as in Theorem \ref{smirnovFormuoliTheorem}, we have the following formula describing their failure to commute with proper pushforward:
    \begin{equation*}
         Q^\bullet(f_!(\alpha)) \smile td_{Q^\bullet}(TP') = f_!(Q^\bullet(\alpha) \smile td_{Q^\bullet}(TP))
    \end{equation*}
    where $\alpha \in H^*_X(P,*)$. 
\end{corollary}

Once one has this formula, one can mimic Brosnan's approach, as in \cite{brosnan2003steenrod} and \cite{brosnan2015comparison}, to define new operations that do commute with pushforward. One simply defines 
\begin{equation*}
    U^\bullet(\alpha) = Q^\bullet(\alpha) \smile td_{Q^\bullet}(TP)
\end{equation*}
again for a smooth variety $P$ with closed support $X$, and $\alpha \in H^*_X(P,*)$. Then, for diagrams as in Theorem \ref{smirnovFormuoliTheorem}, the operations $U^\bullet$ commute due to the Riemann--Roch type formula.

We can now dualize these cohomology operations to obtain homology operations on the higher Chow groups of a (possibly singular) subspace. Of course, this is now purely formal. However, doing so sheds some information on how they relate to the embedding $X \hookrightarrow P$.

\begin{definition}\label{HomologyDualOperations}
    Let $X \hookrightarrow P$ be a closed embedding of a (possibly singular) variety $X$ in a smooth variety $P$. We define operations $U_\bullet$ on the higher Chow groups of $X$ by defining
    \begin{equation*}
        U_s := (\chi\circ i_*)^{-1} \circ U^s \circ (\chi\circ i_*)
    \end{equation*}
    where $\chi\circ i_*: H_i(X,j) \xrightarrow{\simeq} H^{2d-i}_X(P,d-j)$ is Poincar\'e duality composed with localization. Similarly, for $\ell = p$, we can define 
    \begin{equation*}
        Q_s := (\chi\circ i_*)^{-1} \circ Q^s \circ (\chi\circ i_*)
    \end{equation*}
    where $Q^\bullet$ is defined for any prime $\ell$. Similarly, we write $P_s$ as the dual motivic operations.
\end{definition}

We note that the operations $Q_s$ are defined for any prime $\ell$, but for $\ell = p$, we do not have a Riemann--Roch type formula that describes how they interact with proper pushforward. For $\ell \neq p$, and for diagrams as in Theorem \ref{smirnovFormuoliTheorem}, we see that the operations $U_\bullet$ do commute with proper pushforward. We also note that $Q_s: H_i(X,j) \rightarrow H_{i-2s(\ell-1)}(X,\ell j - d(\ell-1))$ where $d$ is the dimension of the ambient space $P$.

\subsubsection{Pushing forward the Bockstein}\label{bocksteinSection}
We include a small section here to show that we can always pushforward the Bockstein operation.

\begin{lemma}
    Let $\beta$ be the Bockstein operation and let $f: X \rightarrow Y$ be a proper map between smooth varieties. Then we have $f_!\beta(x) = \beta f_!(x)$ for any $x \in H^*(X,\Z/\ell(*))$.
\end{lemma}
\begin{proof}
    We write our given proper map into a composition  $f = \pi\circ i$ of a closed embedding and a projection; by Proposition \ref{integrationDecompose} the choice of decomposition does not matter. We check each case individually.

    \textit{The case of a closed embedding:} For a closed embedding $i$, recall by Proposition \ref{embeddingPushforwardIsGivenByThomclass} that we have $i_!(x) = \varphi\circ (\pi^*(x)\smile \tau)$, where $\varphi$ are natural maps from the deformation to the normal cone, $\pi^*$ is the pull-back induced by the projection from the normal bundle, and $\tau$ is the Thom class of the normal bundle. So we check
    \begin{align*}
        \beta i_!(x)
        =& \ \beta (\varphi \circ (\pi^*(x)\smile\tau)) \\
        =& \ \varphi(\beta(\pi^*(x) \smile \tau)) \\
        =& \ \varphi(\beta(\pi^*(x))\smile \tau + \pi^*(x)\smile\beta(\tau)) \\
        =& \ \varphi(\pi^*(\beta x) \smile \tau) = i_!(\beta x)
    \end{align*}
    where we have used that $\beta$ is a derivation, and that it vanishes on the Thom class $\tau$ since $\tau$ is an integral class. That is, $\tau$ is defined with integral coefficients, and so one can reduce it mod $\ell^2$ for any $\ell$, and so from the long exact sequence on cohomology that defines the Bockstein, $\tau$ is in the kernel of $\beta$.

    \textit{The case of a projective bundle:} We have a projection $\pi: Y \times \mathbb{P}^n \rightarrow Y$ that again sends elements $a \otimes t^i$ to $a$ if $i = n$, and $0$ otherwise. So again we check
    \begin{align*}
        \pi_!(\beta (a \otimes t^i))
        =& \ \pi_!(\beta(a) \otimes t^i + a \otimes \beta(t^i)) \\
        =& \ \pi_!(\beta(a) \otimes t^i)
    \end{align*}
    since $\beta$ is a derivation, and vanishes on the integral class $t^i$.
\end{proof}

Once we know we can push forward the Bockstein, we will also abuse notation and write $\beta$ to denote its action on the homology of a subspace $X \hookrightarrow P$, again by pre- and post-composing with Poincar\'e duality. This yields a map $\beta: H_i(X,j) \rightarrow H_{i-1}(X,j)$.

The following proposition is so trivial, we almost did not mention it. However, we present it to demonstrate explicitly why we cannot utilize the motivic operations to define Thom obstructions. See also Theorem 8.4 in \cite{voevodsky2003reduced}.
\begin{proposition}\label{motivicOpsVanishForDegreeReasons}
    The motivic operations $\beta P_s$ vanish on $[X]$ for any closed subvariety $X \hookrightarrow P$.
\end{proposition}
\begin{proof}
    The operations $P_s$ send an element of bidegree $(i,j)$ to bidegree $(i-2s(\ell-1),j-s(\ell-1))$, where again $\ell$ is any prime. The fundamental class of a variety $[X]$ is in bidegree $(2n,n)$ where $n$ is the dimension of $X$, and so $\beta P_s[X]$ lives in bidegree $(2n-2s(\ell-1)-1, n-s(\ell-1))$. In terms of higher Chow groups, this lives in the group $CH_{n-s(\ell-1)}(X, 2n-2s(\ell-1)-1 - (2n-2s(\ell-1))) = CH_{n-s(\ell-1)}(X,-1) = 0$. That is, the group vanishes for trivial degree reasons.
\end{proof}

\section{Algebraic stacks and classifying maps}\label{embeddingBundlesSection}
We now analyze the cohomology of the stack $\text{Vect}_n$ and show that it is concentrated in even degrees over an algebraically closed field, assuming our cohomology for a field vanishes. A previous version of this paper discussed a ``trick'' (as observed by Deligne and Sullivan in \cite{deligne1975fibres}) to bypass the lack of classifying maps for vector bundles over smooth varieties within the category of smooth varieties; this was done by passing to an auxiliary embedding bundle, which had a tautological map to a finite Grassmannian. Unfortunately, the trick only works for cohomology theories given by constant sheaves. As a result, we freely use the theory of stacks. We will see that smoothness of a closed subvariety $X$ inside a smooth variety $P$ gives us a normal bundle, which induces a classifying map of stacks $X \rightarrow \text{Vect}_n$. For references, see \cite{stacks2019stacks} and \cite{neumann2009algebraic}.

\subsection{The classifying stack of vector bundles}\label{deligneSullivanEmbeddingBundletrick}

We let $\textnormal{Vect}_n$ denote the stack of rank $n$ vector bundles and let $\gamma^n$ denote the universal tautological bundle over $\textnormal{Vect}_n$ as usual.

\begin{lemma}
    Let $E$ be a rank $n$ vector bundle over a smooth variety $X$. Then there exists a map of stacks $f: X \rightarrow \textnormal{Vect}_n$ such that $f^*\gamma^n \xrightarrow{\simeq} E$. 
\end{lemma}
Similarly, one has a notion of closed immersions for stacks again defined as those maps that pullback to closed immersions of schemes (see Tags 04YK and 03HB in \cite{stacks2019stacks}). Thus, $\textnormal{Vect}_n$ admits a closed immersion as the zero section into $\gamma^n$, and we can consider the pair of stacks $(\gamma^n,\gamma^n-0)$. We can interpret the stack $\gamma^n-0$ as the stack of bundles equipped with a section that is not the zero section, again with morphisms preserving the section.

Finally, we can define cohomology of a stack. Recall the following categorical fact: every (pre)sheaf is a direct limit of representable (pre)sheaves (see Lemma 7.12.6 of Tag 00WO in \cite{stacks2019stacks}). In other words, every stack is a direct limit of schemes.

\begin{lemma}\label{everypresheafisacolimit}
    Let $\mathcal{C}$ be a site. Let $\mathcal{F}$ be a sheaf of sets on $\mathcal{C}$. Then there exists a diagram $\mathcal{I} \rightarrow \mathcal{C}$, $i \mapsto U_i$, such that $\mathcal{F} = \varinjlim_{\mathcal{I}} h^{\#}_{U_i}$, where $h_{U_i} = Hom(-,U_i)$ for $U_i \in \mathcal{C}$ and $\#$ denotes sheafification.
\end{lemma}
Of course, the indexing for the directed system depends on the stack. For example, $\text{Vect}_n$ is the direct limit of schemes $U$ with maps in the directed system given by the various bundles over $U$ with suitable compatibility constraints. Similarly, the total space $\gamma^n$ is given as the direct limit of the total spaces of all bundles $E$ over all schemes $U$. Moreover, since we only need the category of smooth varieties, we can restrict ourselves to the site $(Sm/k)_{Zar}$ and take direct limits of smooth varieties.

By the above lemma, assuming we have a cohomology theory in the sense of Definition \ref{generalizedcohomologytheory}, we can define cohomology of a stack $\mathcal{S}$ whose domain is the category of smooth varieties.

\begin{definition}
    Let $\mathcal{S}$ be a stack defined on $(Sm/k)_{Zar}$ and let $A$ be a cohomology theory in the sense of Definition \ref{generalizedcohomologytheory}. We define the cohomology of the stack $\mathcal{S}$ as the inverse limit $A(\mathcal{S}) := \varprojlim_{\mathcal{I}} A(U_i)$ where the diagram $\mathcal{I} \rightarrow Sm/k$ with $i \mapsto U_i$ is determined by Lemma \ref{everypresheafisacolimit}.

     Similarly, given any ring operation $\phi$ on a cohomology theory $A$ which is natural with respect to pullback, we can extend the operation $\phi$ to the cohomology of a stack again by taking limits.
\end{definition}

As motivating examples, this defines motivic cohomology and \etale cohomology for $\text{Vect}_n$ as the limits of cohomologies of smooth varieties, and extends the Steenrod operations from the \Einf-algebra structures to the cohomology of a stack. Moreover, we can consider maps of pairs $(E,E-0) \rightarrow (\gamma^n, \gamma^n-0)$ where $E$ is a rank $n$ vector bundle over a smooth variety $X$, and $\gamma^n$ is the universal bundle over $\text{Vect}_n$.  By naturality of the long exact sequence of a pair for a cohomology theory $A$, we can again define the cohomology $A(\gamma^n,\gamma^n-0)$ as the limit of cohomologies $A(E,E-0)$ over all pairs $(E,E-0)$. 

\begin{lemma}\label{thomisomorphismforuniversalbundle} (Thom isomorphism for the universal bundle)
    Let $A$ be a bigraded cohomology theory equipped with a Thom structure. We have additive isomorphisms $A^{*+2n,*+n}(\gamma^n,\gamma^n-0) \cong A^{*,*}(\textnormal{Vect}_n)$.
\end{lemma}
\begin{proof}
    By the Thom isomorphism for each pair $A^{*+2n,*+n}(E,E-0)$ where $E$ is a rank $n$ bundle over $X$, the cohomology groups of the pair $(\gamma^n,\gamma^n-0)$ are isomorphic to the inverse limit of the groups $A^{*,*}(X)$ of the various base spaces $X$, each admitting a map to $\text{Vect}_n$ corresponding to the bundle $E$. But $A^{*,*}(\text{Vect}_n)$ is defined as the limit of the groups $A^{*,*}(X)$.  
\end{proof}

Thus, by definition, for every bundle $E$ over a smooth variety $X$, we have a diagram

\begin{center}
\begin{tikzcd}
E \arrow[r] \arrow[d]
& \gamma^n \arrow[d] \\
U \arrow[r, "u"]
& \text{Vect}_n
\end{tikzcd}
\end{center}

which induces a map of pairs $(E,E-0) \rightarrow (\gamma^n,\gamma^n-0)$ and in turn induces a map on cohomology $A(\gamma^n,\gamma^n-0) \rightarrow A(E,E-0)$. Moreover, this map on cohomology commutes with any natural ring operation $\phi$ defined on cohomology of a pair. It remains then to compute the cohomology groups of $(\gamma^n,\gamma^n-0)$, which by Lemma \ref{thomisomorphismforuniversalbundle} is equivalent to computing the cohomology groups $A(\text{Vect}_n)$.

Finally, we consider the infinite Grassmannian $Gr(n,\infty)$, which, as a stack (in fact an ind-scheme), is defined as the direct limit of finite Grassmannians $Gr(n,N)$ as $N$ goes to infinity. $Gr(n,\infty)$ also has a universal bundle $\gamma^n$, which classifies globally generated vector bundles. We now defer to Morel and Voevodsky (Proposition 3.7 in \cite{morel19991}), who prove that $\text{Vect}_n$ and the infinite Grassmannian $Gr(n,\infty)$ are equivalent in the $\mathbb{A}^1$ homotopy category. In particular, for any cohomology theory $A$ in the sense of Definition \ref{generalizedcohomologytheory}, we have the following proposition.

\begin{proposition}
    For any bigraded ringed cohomology theory $A$ in the sense of Definition \ref{generalizedcohomologytheory}, we have an isomorphism of bigraded rings $A^{*,*}(\textnormal{Vect}_n) \cong A^{*,*}(Gr(n,\infty))$, where $Gr(n,\infty)$ is the infinite Grassmannian.  
\end{proposition}

We note here that the infinite Grassmannian is \textbf{not} isomorphic as a stack to $\text{Vect}_n$. The latter classifies all vector bundles, while the former classifies vector bundles that are globally generated. In the topological and smooth categories, this distinction does not appear, but it is relevant in the holomorphic category (as not every short exact sequence of holomorphic bundles splits). However, the point is that they are cohomologically equivalent.

\subsection{Cohomology of the infinite Grassmannian}
In this section, we will analyze the cohomological structure of the infinite-dimensional Grassmannian. The arguments here are standard in algebraic geometry. We present the weakest form necessary for our result. See also Lemmas 5, 6, and 7 in \cite{laksov1972algebraic}, and Example 14.6.2 in \cite{fulton2013intersection}.

\begin{proposition}
    Let $A$ be a bigraded cohomology theory in the sense of Panin--Smirnov. Assume $A^{i,j}(\textnormal{Spec }k) = 0$ for $i \neq 0$, and let $Gr(m,N)$ denote the Grassmannian of $m$-dimensional subspaces in $\mathbb{A}^N$. Then $A^{i,j}(Gr(m,N)) = 0$ when $i$ is odd.
\end{proposition}
\begin{proof}
    We proceed by induction on $m$ and $N$. If $m = 1$, then $Gr(m,N) = \mathbb{P}^N$, and by the assumption on our field $k$ and by the projective bundle theorem, the result follows.

    Suppose we have the result for all $m \leq N$. Write $\mathbb{A}^{N+1} = \mathbb{A}^N \times \mathbb{A}^1$. Then there is a closed embedding $Gr(m,N) \hookrightarrow Gr(m+1,N+1)$ which maps a subspace $C$ to $C \oplus \mathbb{A}^1$. Let us write the complement of this embedding as $U$. By Lemma 7 in \cite{laksov1972algebraic}, the open complement $U$ is an affine bundle over $Gr(m+1,N)$. By the localization sequence, we then obtain
    \begin{equation*}
        \cdots \rightarrow A^{i-2c,j-c}(Gr(m,N)) \rightarrow A^{i,j}(Gr(m+1,N+1)) \rightarrow A^{i,j}(U) \rightarrow \cdots
    \end{equation*}
    where $c$ is the codimension of $Gr(m,N)$ in $Gr(m+1,N+1)$. Moreover, since $U$ is an affine bundle, by homotopy invariance we can identify the last term with $A^{i,j}(Gr(m+1,N))$. By the inductive step, the two outer terms in the three term sequence above both vanish when $i$ is odd. The result follows immediately.
\end{proof}


\begin{corollary}
    Let $A$ be a bigraded cohomology theory in the sense of Panin--Smirnov. Assume $A^{i,j}(\textnormal{Spec }k) = 0$ for $i \neq 0$, and let $Gr(n,\infty)$ denote the infinite Grassmannian of $n$-planes. Then $A^{i,j}(Gr(n,\infty)) = 0$ when $i$ is odd. Moreover, $A^{i,j}(\gamma^n,\gamma^n-0) = 0$ when $i$ is odd, where $\gamma^n$ denotes the universal bundle over $\textnormal{Vect}_n$.
\end{corollary}

\section{Thom obstructions to resolution of singularities}\label{ThomObsVanish}
\subsection{Vanishing for smooth varieties}
We are now ready to present the motivic analog of the classical argument due to Thom.

\begin{theorem}
Let $X \hookrightarrow P$ be a closed embedding with $X$ and $P$ both smooth, over an algebraically closed field. Let $Q^\bullet$ be the simplicial Steenrod operations on $H^*_X(P,\Z/\ell(*))$, where $\ell$ is any prime, and let $Q_\bullet$ denote the dual operations on $H_*(X,\Z/\ell(*))$. The odd degree operations $\beta Q_s$ vanish on the fundamental class $[X]$. 
\end{theorem}
\begin{proof}
    The statement is equivalent to the odd degree operations $\beta Q^s$ vanishing on the mod $\ell$ dual class $\tau \in H^{2c}_X(P,\Z/\ell(c))$, where $c$ is the codimension of $X$ in $P$. By deformation to the normal cone, we have an isomorphism $H^*_X(P,\Z/\ell(*)) \cong H^{*}(N_XP, N_XP-X,\Z/\ell(*))$ where $N_XP$ denotes the normal bundle of $X$ in $P$. Moreover, as the maps in the deformation to the normal cone are all between smooth varieties, the operations $\beta Q^s$ commute with the given isomorphism of cohomology groups.

    Now we have the commutative diagram
    \begin{center}
\begin{tikzcd}
N_XP \arrow[r] \arrow[d]
& \gamma^c \arrow[d] \\
X \arrow[r]
& \text{Vect}_n
\end{tikzcd}
\end{center}
and again by functoriality of the Thom isomorphism, we see that there is a cohomology class $\gamma$ in $H^*(\gamma^c,\gamma^c-0,\Z/\ell(*))$ that is mapped to $\tau$. Specifically, $\gamma$ is in bidegree $(2c,c)$, where $c$ is the codimension of $X$ in $P$. The operations $\beta Q^s$ then vanish on $\gamma$ since $Q^s$ keeps the first index even, but $\beta$ increases it by $1$. Since $Vect_n$ has vanishing cohomology when the first index is odd, we then have $\beta Q^s(\gamma) = 0$. By naturality of the operations, these pullback to give that $\beta Q^s(\tau) = 0$. Passing through deformation to the normal cone, we conclude that $\beta Q^s$ vanish on the dual class of $X$, and so the dual homology operations vanish on $[X]$.
\end{proof}

Let us present the version for generalized cohomology theory. The proof is exactly the same.

\begin{theorem}\label{vanishingTheoremforSmooth}
    Let $X \hookrightarrow P$ be a closed embedding with $X$ and $P$ both smooth. Let $A$ be a bigraded ring cohomology theory in the sense of Panin--Smirnov, which is equipped with a proper pushforward. Suppose $A^{i,j}(\textnormal{Spec } k)$ vanishes for $i \neq 0$. Let $\phi = \phi^\bullet$ be a ring operation that preserves even degrees (in the first index) and let $\beta$ be a natural operation that sends bidegrees $(i,j)$ to $(i+1,j)$. The operation $\beta\phi^s$ vanishes on the class $\tau \in A^{2c,c}_X(P)$, where $c$ is the codimension of $X$ in $P$, and $\tau$ corresponds to the Thom class of the normal bundle of $X$ in $P$.
\end{theorem}

\subsection{Vanishing for singular varieties with a resolution}
We have seen a vanishing result for smooth varieties. We now extend the result to singular varieties that admit a resolution of singularities. Recall given a ring operation $\phi$ on a bigraded ring cohomology theory $A$ in the sense of Panin--Smirnov, where $\phi$ has invertible Todd genus on $A(\mathbb{P}^\infty)$, we can define a new operation $\tilde{\phi}$ by
\begin{equation*}
    \tilde{\phi}(\alpha) = \phi(\alpha) \smile td_{\phi}(TP)
\end{equation*}
where $\alpha \in A_X(P)$. The operation $\tilde{\phi}$ commutes with proper pushforward by the Riemann--Roch type formula we have discussed in Section \ref{RiemannRochTypeFormulasSection}.

\begin{theorem}\label{Aprioriresolutionobstructions}
    Let $A$ be a bigraded ring cohomology theory in the sense of Panin--Smirnov, equipped with a proper pushforward. Suppose $A^{i,j}(\textnormal{Spec }k) = 0$ for $i \neq 0$. Let $\phi = \phi^\bullet$ be a ring operation with well-defined Todd genus that preserves even degrees in the first index, and let $\beta$ be a natural operation that sends bidegrees $(i,j)$ to $(i+1,j)$ that commutes with proper pushforward. Moreover, assume $\beta$ is a derivation that vanishes on the Todd genus of any bundle. Let $X' \hookrightarrow P'$ be a singular variety with a closed embedding into a smooth variety $P'$. Suppose the pair $(P',X')$ admits an embedded resolution of singularities. The operation $\beta\tilde{\phi}^s$ vanishes on the dual class $\tau' \in A^{2c,c}_{X'}(P')$ where $c$ is the codimension of $X'$ in $P'$, and where $\tilde{\phi}$ is the operation associated to $\phi$ that commutes with proper pushforward.
\end{theorem}
\begin{proof}
    An embedded resolution of singularities is a diagram of the form
    \begin{center}
\begin{tikzcd}
P \arrow[r, "f"]
& P' \\
X \arrow[r, "f'"] \arrow[hookrightarrow, u,"i"]
& X' \arrow[hookrightarrow, u, "i'"].
\end{tikzcd}
\end{center}
where $X$ and $P$ are smooth, and the maps $f$ and $f'$ are proper, with $f'$ a birational isomorphism. By Theorem \ref{vanishingTheoremforSmooth}, the operation $\beta\phi$ vanishes on $\tau \in A^{2c,c}_X(P)$ the Thom class of the normal bundle of $X$ in $P$. We then see that
\begin{equation*}
\beta\tilde{\phi}(\tau) = \beta (\phi(\tau) \smile td_{\phi}(TP)) = \beta\phi(\tau) \smile td_{\phi}(TP) + \phi(\tau)\smile \beta td_{\phi}(TP) = 0
\end{equation*}
since $\beta\phi(\tau) = 0$ and by our assumptions on $\beta$. we see that the proper pushforward $f_!$ maps $\tau$ to $\tau'$, since $f$ is generically degree $1$ and by functoriality of the Thom structure. Thus, we conclude that
\begin{equation*}
    \beta\tilde{\phi}(\tau') = \beta\tilde{\phi}(f_!(\tau)) = f_!(\beta\tilde{\phi}(\tau)) = 0
\end{equation*}
since all operations above commute with proper pushforward.
\end{proof}

Thus, if $X'$ is a singular variety with closed embedding into a smooth variety $P'$, and $X'$ admits an embedded resolution of singularities, then the mod $\ell \neq p$ operations $U_\bullet$ from Definition \ref{HomologyDualOperations} vanish on its fundamental class $[X']$. At first, this may seem like a necessary condition for a variety to admit a resolution of singularities; however, we will see later that resolution of singularities is not needed to prove this vanishing result. Unfortunately, all the examples the author knows (such as motivic mod $\ell \neq p$, and \etale mod $\ell \neq p$), vanishes for other a priori reasons other than resolution of singularities.

Let us see what happens for the operations $Q^\bullet$ defined mod $\ell = p$. Suppose we have an embedded resolution of singularities $f: P \rightarrow P'$, with a diagram as in Theorem \ref{smirnovFormuoliTheorem}. We can again write our proper map $f$ as a composition $f = \pi \circ i$ where $i$ is a closed embedding $i = f \times e: P \rightarrow P' \times \mathbb{P}^n$, where $e$ is some embedding of $P$ into $\mathbb{P}^n$. We again write the dual class of $X$ in $P$ as $\tau$, and the dual class of $X'$ in $P'$ as $\tau'$. Since the map $f$ is generically degree $1$, we see that $f_*[X] = [X']$, and so $f_!(\tau) = \tau'$. We also have $i_!(\tau) = (f \times e)_!(\tau) = \tau' \otimes t^n$, by definition of proper pushforward for higher Chow groups.

By Theorem \ref{vanishingTheoremforSmooth}, we see that $\beta Q^\bullet(\tau) = 0$. On one hand, we see that 
\begin{equation*}
    \beta Q^\bullet(\tau') = \beta Q^\bullet(f_!(\tau)) = \beta Q^\bullet (\pi_! \circ i_!)(\tau) = \beta Q^\bullet (\pi_!(\tau' \otimes t^n))
\end{equation*}
and we would like to commute $\beta Q^\bullet$ with the proper pushforward $\pi_!$ from the projective bundle $P' \times \mathbb{P}^n \xrightarrow{\pi} P'$. But on the other hand,
\begin{equation*}
    \pi_!(\beta Q^\bullet(\tau' \otimes t^n)) = \pi_!(\beta (Q^\bullet(\tau') \otimes Q^\bullet(t^n))) = 0
\end{equation*}
since $\beta(Q^\bullet(\tau')\otimes Q^\bullet(t^n)) = \beta Q^\bullet(\tau')\otimes Q^\bullet(t^n) + Q^\bullet(\tau')\otimes \beta Q^\bullet(t^n)$ as $\beta$ is a derivation, and $Q^\bullet(t^n) = 0$ for degree reasons (since $Q^0$ mod $p$ is \textit{not} the identity!). Thus, we unfortunately cannot say anything about $\beta Q^\bullet(\tau')$ for singular varieties $X'$ in smooth $P'$, as the operation $Q^\bullet$ does not have a well-defined Todd genus. Recall that the argument for mod $\ell \neq p$ worked since one could check that $Q^0$ on the cohomological generator $t$ of $\mathbb{P}^n$ was indeed the identity, and so all formulas apply.

\section{Vanishing of mod \texorpdfstring{$\ell \neq p$}{} obstructions}\label{alterationsSection}

In this section we show the mod $\ell \neq p$ obstructions we have defined are not obstructions to resolution of singularities at all. That is, they vanish without using resolution of singularities.

\begin{theorem}
    The mod $\ell \neq p$ operations $\beta U_\bullet$ from Definition \ref{HomologyDualOperations} vanish on $[X']$ for a singular variety $X'$ with closed embedding into a smooth variety $P'$. 
\end{theorem}
\begin{proof}
    Every singular variety $X'$ admits an alteration $X$ by work of de Jong \cite{de1996smoothness}. Moreover, by work of Temkin \cite{temkin2017tame}, one can find an alteration $X$ of degree equal to powers of $p$, the characteristic. Let us now assume $X'$ has an embedded alteration $X$ where the map $f$ has degree equal to $p^m$ for some $m$.

    By Theorem \ref{vanishingTheoremforSmooth}, the operations $\beta U_\bullet$ vanish on $[X]$. Then we see that
    \begin{equation*}
        0 = f_!(\beta U_\bullet[X]) = \beta U_\bullet f_!([X]) = p^m \beta U_\bullet [X']
    \end{equation*}
    but since the operations are defined mod $\ell \neq p$, we see that $p^m$ is invertible. So we conclude $\beta U_\bullet[X'] = 0$.
\end{proof}

The above argument then proves the following modification of Theorem \ref{Aprioriresolutionobstructions}. Essentially, if the cohomology theory $A$ takes values mod $\ell \neq p$, then we can drop the assumption on the singular variety $X'$ having a resolution of singularities.

\begin{theorem}
    Let $A$ be a bigraded ring cohomology theory in the sense of Panin--Smirnov, valued in modules over $\Z/\ell$ for $\ell \neq p$, and equipped with a proper pushforward. Suppose $A^{i,j}(\textnormal{Spec }k) = 0$ for $i \neq 0$. Let $\phi = \phi^\bullet$ be a ring operation with well-defined Todd genus that preserves even degrees in the first index, and let $\beta$ be a natural operation that sends bidegrees $(i,j)$ to $(i+1,j)$ that commutes with proper pushforward. Moreover, assume $\beta$ is a derivation that vanishes on the Todd genus of any bundle. Let $X' \hookrightarrow P'$ be a singular variety with a closed embedding into a smooth variety $P'$. The operation $\beta\tilde{\phi}^s$ vanishes on the class $\tau' \in A^{2c,c}_{X'}(P')$ where $c$ is the codimension of $X'$ in $P'$, where $\tau'$ is the Thom class of the normal bundle of $X'$ in $P'$, and where $\tilde{\phi}$ is the operation associated to $\phi$ that commutes with proper pushforward.
\end{theorem}

In other words, the mod $\ell \neq p$ simplicial operations on motivic and \etale cohomology all vanish, without using resolution of singularities. In fact, the mod $\ell \neq p$ simplicial operations vanish for simpler reasons without even using alterations, by the following theorem of Brosnan and Joshua (Proposition 6.1 in \cite{brosnan2015comparison}).

\begin{theorem}
    Let $k$ have a primitive $\ell$-th root of unity, for $\ell \neq p$. Let $X$ be a variety with closed embedding into a smooth variety $P$. Then
    \begin{equation*}
        Q^s(\alpha) = B^{(q-s)(\ell-1)}P^s(\alpha), \beta Q^s(\alpha) = B^{(q-s)(\ell-1)}\beta P^s(\alpha)
    \end{equation*}
    for $s \leq q$, where $B$ denotes the motivic Bott element, and where $P^s$ denotes the motivic operations, for $\alpha \in H^i_X(P,\Z/\ell(q))$, with $i \leq 2q$. For $s > q$ both operations are zero.
\end{theorem}
Recall that the motivic operations $\beta P^\bullet$ vanish on fundamental classes of varieties for trivial degree reasons (Proposition \ref{motivicOpsVanishForDegreeReasons}). And so the mod $\ell$ simplicial operations vanish similarly, since they differ only by multiplication by an invertible element.

We also remark that the only mod $\ell = p$ operations we know that do push forward are those given by the motivic operations $P^\bullet$. But of course they again vanish for trivial degree reasons (Proposition \ref{motivicOpsVanishForDegreeReasons}).

\bibliographystyle{alpha}
\bibliography{bibliography.bib}

\end{document}